\documentclass[12pt]{amsart}
\usepackage[margin=0.9in]{geometry}
\usepackage{amscd, amssymb, amsmath, wasysym}
\usepackage{graphicx}
\usepackage{amsfonts}
\usepackage{mathrsfs}    
\usepackage{mathtools}
\usepackage{eucal}     
\usepackage{latexsym}   
\usepackage{verbatim}   
\usepackage[all]{xy}   
\usepackage[dvipsnames]{xcolor}
\usepackage{bookmark}

\usepackage{hyperref}
 \hypersetup{
     colorlinks=true,
     linkcolor=NavyBlue,
     filecolor=NavyBlue,
     citecolor = TealBlue,
     urlcolor=magenta,
  }
  
\makeatletter  

\newtheorem{theorem}{Theorem}[section]
\newtheorem{lemma}[theorem]{Lemma}
\newtheorem{proposition}[theorem]{Proposition}

\newtheorem{question}[theorem]{Question}

\newtheorem{corollary}[theorem]{Corollary}

\theoremstyle{definition} % italic or bold etc.
\newtheorem{definition}[theorem]{Definition}
\newtheorem{definition-lemma}[theorem]{Definition-Lemma}

\newtheorem{example}[theorem]{Example}

\theoremstyle{remark}
\newtheorem{remark}[theorem]{Remark}

\numberwithin{equation}{section}

\newcommand{\nC}{\mathbf{C}}

\def\nP{\mathbf{P}}
\def\nT{\mathbf{T}}
\newcommand{\sB}{\mathscr{B}}
\newcommand{\sC}{\mathscr{C}}
\newcommand{\sO}{\mathscr{O}}
\newcommand{\sI}{\mathscr{I}}
\newcommand{\sE}{\mathscr{E}}

\def\Pic{\operatorname{Pic}}

\def\mult{\operatorname{mult}}

\def\Cone{\operatorname{Cone}}

\def\Bs{\operatorname{Bs}}

\def\Supp{\operatorname{Supp}}

\def\Sing{\operatorname{Sing}}
\def\length{\operatorname{length}}

\def\out{\operatorname{out}}
\def\inn{\operatorname{inn}}

\def\Ram{\operatorname{Ram}}

\newcommand*{\op}[1]{\operatorname{#1}}%

\renewcommand*{\emptyset}{\varnothing}
\newcommand*{\pd}{\partial}

\newcommand*{\xto}[1]{\xrightarrow{\,#1\,}}
\makeatletter
\newcommand*{\da@rightarrow}{\mathchar"0\hexnumber@\symAMSa 4B }
\newcommand*{\da@leftarrow}{\mathchar"0\hexnumber@\symAMSa 4C }
\newcommand*{\xdashrightarrow}[2][]{%
  \mathrel{%
	\mathpalette{\da@xarrow{#1}{#2}{}\da@rightarrow{\,}{}}{}%
  }%
}
\newcommand{\xdashleftarrow}[2][]{%
  \mathrel{%
	\mathpalette{\da@xarrow{#1}{#2}\da@leftarrow{}{}{\,}}{}%
  }%
}
\newcommand*{\da@xarrow}[7]{%
  \sbox0{$\ifx#7\scriptstyle\scriptscriptstyle\else\scriptstyle\fi#5#1#6\m@th$}%
  \sbox2{$\ifx#7\scriptstyle\scriptscriptstyle\else\scriptstyle\fi#5#2#6\m@th$}%
  \sbox4{$#7\dabar@\m@th$}%
  \dimen@=\wd0 %
  \ifdim\wd2 >\dimen@
	\dimen@=\wd2 %   
  \fi
  \count@=2 %
  \def\da@bars{\dabar@\dabar@}%
  \@whiledim\count@\wd4<\dimen@\do{%
	\advance\count@\@ne
	\expandafter\def\expandafter\da@bars\expandafter{%
	  \da@bars
	  \dabar@ 
	}%
  }%  
  \mathrel{#3}%
  \mathrel{%   
	\mathop{\da@bars}\limits
	\ifx\\#1\\%
	\else
	  _{\copy0}%
	\fi
	\ifx\\#2\\%
	\else
	  ^{\copy2}%
	\fi
  }%   
  \mathrel{#4}%
}
\makeatother

\title[Positivity of double point divisors]
{Positivity of double point divisors}

\begin{document}

\author{Yonghwa Cho}
\address{Department of Mathematics, Gyeongsang National University, 501 Jinju-daero, Jinju 52828, Gyeongsangnam-do, Republic of Korea}
\email{yhcho@gnu.ac.kr}

\author{Jinhyung Park}
\address{Department of Mathematical Sciences, KAIST, 291 Daehak-ro, Yuseong-gu, Daejeon 34141, Republic of Korea}
\email{parkjh13@kaist.ac.kr}

%\thanks{}

\thanks{Y. Cho was supported by the research grant of the new professor of the Gyeongsang National University in 2024 (GNU-2024-240034). J. Park was partially supported by the National Research Foundation (NRF) funded by the Korea government (MSIT) (NRF-2019R1A6A1A10073887).} %NRF-2022M3C1C8094326

\date{\today}

\begin{abstract}
The non-isomorphic locus of a general projection from an embedded smooth projective variety to a hypersurface moves in a linear system of an effective divisor which we call the double point divisor. David Mumford proved that the double point divisor from outer projection is always base point free, and Bo Ilic proved that it is ample except for a Roth variety. The first aim of this paper is to show that the double point divisor from outer projection is very ample except in the Roth case. This answers a question of Bo Ilic. Unlike the case of outer projection, the double point divisor from inner projection may not be base point free nor ample. However, Atsushi Noma proved that it is semiample except when a variety is neither a Roth variety, a scroll over a curve, nor the second Veronese surface. In this paper, we investigate when the double point divisor from inner projection is base point free or big.
\end{abstract}

\maketitle
%\tableofcontents

%%%%%%%%%%%%%%%%%%%%%%%%%%%%%%%%%%%%%%%%%%%%%%%%%%%%%%%
\section{Introduction}
%%%%%%%%%%%%%%%%%%%%%%%%%%%%%%%%%%%%%%%%%%%%%%%%%%%%%%%

Throughout the paper, we work over the field $\nC$ of complex numbers, $X \subseteq \nP^r$ is a non-degenerate smooth projective variety of dimension $n$, codimension $e$, and degree $d$, and $H$ is its hyperplane section. Here the non-degeneracy condition includes $X \neq \nP^r$. For a general $(e-1)$-dimensional linear subspace $\Lambda \subseteq \nP^r$, consider the projection
$$
\pi_{\Lambda, X} \colon X \longrightarrow \nP^n
$$
centered at $\Lambda$.
The ramification locus $\Ram(\pi_{\Lambda, X})$ of $\pi_{\Lambda, X}$ is an effective divisor linearly equivalent to $K_X+(n+1)H$. Mumford proved that $K_X+(n+1)H$ is base point free (\cite[Technical Appendix 4]{BM}), and Ein proved that $K_X + nH$ is base point free (\cite[Lemma 2]{E}). By \cite[Theorem 1.4]{Io1}, $K_X + (n-1)H$ is not base point free if and only if $X=\nP^1$ or $(X, \sO_X(1)) = (Q^n, \sO_{Q^n}(1)), (\nP(\sE), \sO_{\nP(\sE)}(1)), (\nP^2, \sO_{\nP^2}(2))$, where $Q^n \subseteq \nP^{n+1}$ is a quadric hypersurface and $\sE$ is a very ample vector bundle of rank $\geq 2$ on a smooth projective curve. Suppose for a moment that $K_X+(n-1)H$ is base point free. The \emph{adjunction mapping} $\varphi \colon X \to B$ is a surjective morphism given by the linear system $\lvert K_X+(n-1)H \rvert$. A basic theorem in adjunction theory (cf. \cite[Proposition 1.11]{Io1}) states that $\dim B = n$ (in particular, $K_X + (n-1)H$ is big) except in the following cases:
\begin{enumerate}
    \item $\dim B = 0$: $X$ is a del Pezzo manifold. In this case, $K_X+(n-1)H=0$.
    \item $\dim B = 1$: $\varphi$ is a hyperquadric fibration over a smooth projective curve $B$.
    \item $\dim B = 2$: $\varphi$ is a linear fibration over a smooth projective surface $B$.
\end{enumerate}

\medskip

From now on, we assume that $e \geq 2$. For an $(e-2)$-dimensional linear subspace $\Lambda \subseteq \nP^r$ with $X \cap \Lambda = \emptyset$, consider the outer projection
$$
\pi_{\Lambda, X} \colon X \longrightarrow \overline{X} \subseteq \nP^{n+1}
$$
centered at $\Lambda$. Note that $\pi_{\Lambda, X}$ is a finite birational morphism for a general choice of $\Lambda$. The non-isomorphic locus $D_{\out}(\pi_{\Lambda, X})$ of $\pi_{\Lambda, X}$ is an effective divisor linearly equivalent to the \emph{double point divisor from outer projection}
$$
D_{\out}:=-K_X + (d-n-2)H.
$$
It is an interesting problem to study the positivity of double point divisors as for ramification divisors. Mumford proved that $D_{\out}$ is base point free (\cite[Technical Appendix 4]{BM}), and Ilic proved that $D_{\out}$ is ample if and only if $X \subseteq \nP^r$ is not a Roth variety (\cite[Theorem 4.2]{Il}).\footnote{In \cite{Il}, the linearly normality condition is added in the definition of a Roth variety, but in this paper, this condition is dropped.} 

\medskip

In \cite[Remark 5.8]{Il}, Ilic asks whether $D_{\out}$ is very ample unless $X \subseteq \nP^r$ is a Roth variety. A member $D$ of $\lvert D_{\out} \rvert$ is called a \emph{geometric divisor} if $D = D_{\out}(\pi_{\Lambda, X})$ for some $(e-2)$-dimensional linear subspace $\Lambda \subseteq \nP^r$ such that the projection $\pi_{\Lambda, X}$ is finite and birational. Let $S_{\out}$ be the set of the sections $s \in H^0(X, D_{\out})$ such that $\operatorname{div}(s)$ is a geometric divisor (such a section $s$ is said to be \emph{geometric}), and $V_{\out}$ be the subspace of $H^0(X, D_{\out})$ spanned by $S_{\out}$. Mumford actually proved that the linear system $\lvert V_{\out} \rvert$ is base point free by showing that for any $x \in X$, there exists $s \in S_{\out}$ such that $s(x) \neq 0$. Ilic precisely proved that when $X \subseteq \nP^r$ is not a Roth variety, $\lvert V_{\out} \rvert$ separates two distinct points of $X$ by showing that for any $x,y \in X$ with $x \neq y$, there exists $s \in S_{\out}$ such that $s(x)=0$ but $s(y) \neq 0$. Ilic's question is equivalent to the separation of the tangent directions by the sections in $H^0(X, D_{\out})$. As Ilic pointed out in \cite[Remark 5.8]{Il}, it would be more interesting to understand when it is possible with geometric sections in $S_{\out}$ or their linear combinations in $V_{\out}$.

\medskip

The first aim of this paper is to provide a complete answer to all aspects of Ilic's question. For this purpose, we need to consider the \emph{inner Segre locus}
$$
\mathscr{C}(X):=\{u \in X \mid \length(X \cap \langle u, x \rangle) \geq 3~\text{for general $x \in X$} \}.
$$
We refer to \cite{N0, N1, N2} for more basic properties of Segre loci.

\begin{theorem}\label{thm:D_outva}
Let $X \subseteq \nP^r$ be a non-degenerate smooth projective variety such that it is not a hypersurface. 
\begin{enumerate}
\item For a point $x \in X$, the set $S_{\out}$ of geometric sections separates tangent directions at $x$ if and only if $x \not\in \sC(X)$. In particular, if $\sC(X) = \emptyset$, then $S_{\out}$ separates all the tangent directions.
\item The subspace $V_{\out}$ of $H^0(X, D_{\out})$ spanned by geometric sections separates all the tangent directions. In particular, if $X \subseteq \nP^r$ is not a Roth variety, then the linear system $\lvert V_{\out} \rvert$ is very ample, and hence, the divisor $D_{\out}$ is very ample. 
\end{enumerate}
\end{theorem}

Note that a subset $S$ of $H^0(X, D_{\out})$ separates tangent directions at $x \in X$ if and only if for any projective tangent line $L$ to $X$ at $x$, there is $s \in S$ such that $L \not\subseteq \nT_x \operatorname{div}(s)$. The proof of Theorem \ref{thm:D_outva} $(1)$ goes as follows. If a geometric divisor $D_{\out}(\pi_{\Lambda, X})$ passes through $x \in \sC(X)$, then $\length(\pi_{\Lambda,X}^{-1}(\pi_{\Lambda,X}(x))) \geq 3$. This yields $\mult_x D_{\out}(\pi_{\Lambda, X}) \geq 2$, so $\nT_x D_{\out}(\pi_{\Lambda, X}) = \nT_x X$. Conversely, if $x \not\in \sC(X)$, then for any projective tangent line $L$ to $X$ at $x$, we can construct a geometric divisor $D_{\out}(\pi_{\Lambda, X})$ such that $L \not\subseteq \nT_x D_{\out}(\pi_{\Lambda, X})$. Note that $\sC(X) = \emptyset$ in many cases such as $X \subseteq \nP^r$ is contained in Veronese embeddings or a scroll over a smooth projective curve (\cite[Corollary 3]{N0}, \cite[Theorem 1]{N4}). In this case, $\lvert V_{\out} \rvert$ separates all the tangent directions. If $\sC(X) \neq \emptyset$, then $X$ has a special structure (see \cite[Theorem 2]{N1}, \cite[Theorem 4]{N2}).

\medskip

The proof of Theorem \ref{thm:D_outva} $(2)$ is more involved. For a point $x \in \sC(X)$ and a projective tangent line $L$ to $X$ at $x$, we need to find $s_1, s_2 \in S_{\out}$ such that $s_1(x), s_2(x) \neq 0$ and $L \not\subseteq \nT_x \operatorname{div}(a_1s_1-a_2s_2)$ for some $a_1, a_2 \in \nC$ with $(a_1s_1 - a_2s_2)(x) = 0$. It is very hard to study the geometry of $\operatorname{div}(a_1s_1-a_2s_2)$ in general. Our strategy to overcome this difficulty is to use Mumford's observation in \cite[Technical Appendix 4]{BM}: if $\Lambda' \subseteq \nP^r$ is a general $(e-1)$-dimensional linear subspace and $\Lambda \subseteq \Lambda'$ is a general subspace of codimension $1$, then 
\begin{equation}\label{eq: Mumford identity, abstract form}
D_{\out}(\pi_{\Lambda, X}) + \Ram(\pi_{\Lambda', X}) = V(F_{\Lambda, \Lambda'}) \cap X,
\end{equation}
where $f_{\Lambda} \in \nC[Z_0, \ldots, Z_{n+1}]$ is a homogeneous polynomial with $\pi_{\Lambda,X}(X)=V(f_{\Lambda}) \subseteq \nP^{n+1}$ and
$F_{\Lambda, \Lambda'} = \sum_{i=0}^{n+1} q_i (\partial f_{\Lambda}/\partial Z_0)$ for $\pi_{\Lambda, \nP^r}(\Lambda')=[q_0, \ldots, q_{n+1}] \in \nP^{n+1}$ is regarded as a form on $\nP^r$. Let $\ell_{\Lambda, \Lambda'}:=\sum_{j=0}^{n+1} (\partial F_{\Lambda, \Lambda'}/\partial Z_j)(x) Z_j$
be a linear form on $\nP^r$. Then $\ell_{\Lambda, \Lambda'}(x)=0$ if and only if $x \in V(F_{\Lambda, \Lambda'})$ and $V(\ell_{\Lambda, \Lambda'}) = \nT_x V(F_{\Lambda, \Lambda'})$.
It suffices to show that there are general $(e-2)$-dimensional linear subspaces $\Lambda_1, \Lambda_2$ of $\Lambda'$ such that 
$V(\ell_{\Lambda_1, \Lambda'})$ and $V(\ell_{\Lambda_2, \Lambda'})$ meet $L$ at distinct points different from $x$. In this case, there are $a_1, a_2 \in \nC$ with $V(a_1 \ell_{\Lambda_1, \Lambda'}-a_2\ell_{\Lambda_2, \Lambda'}) \cap L = \{x\}$, i.e., $x \in V(a_1 F_{\Lambda_1, \Lambda'}-a_2 F_{\Lambda_2, \Lambda'})$ and $L \not\subseteq \nT_x V(a_1 F_{\Lambda_1, \Lambda'}-a_2 F_{\Lambda_2, \Lambda'})$. One can reduce the problem to the space curve case. We study the space $\mathcal{L}_{P, T}$ of points $Q \in \nP^3$ such that $\ell_{P , \langle P, Q \rangle}(T)=0$ for general $P \in \nP^3$ and $T \in L$. If $P' \in \mathcal{L}_{P,T}$ but $P \not\in \mathcal{L}_{P',T}$ for general $P, P' \in \nP^3$, then we are done. If not and $x$ is not an inflection point, then $\bigcup_{T \in \nT_x X} (\mathcal{L}_{P,T} \cap \mathcal{L}_{P',T})$ is the union of two planes $\langle P, P', x \rangle$ and $\Theta$. We then show that there is a general point $P'' \in \Theta$ such that $P \in \mathcal{L}_{P'',T}$ but $P'' \not\in \mathcal{L}_{P,T}$. The most difficult part is dealing with the case when $x$ is an inflection point or $\Theta$ is an osculating plane. This requires very technical and sophisticated calculations.

\medskip

For general $x_1, \ldots, x_{e-1} \in X$ and $\Lambda :=\langle x_1, \ldots, x_{e-1} \rangle \subseteq \nP^r$, consider the inner projection
$$
\pi_{\Lambda, X} \colon X \dashrightarrow \overline{X} \subseteq \nP^{n+1}
$$
centered at $\Lambda$. Noma \cite[Theorem 4]{N1} proved that the divisorial part $D_{\inn}(\pi_{\Lambda, X})$ of the non-isomorphic locus of $\pi_{\Lambda, X}$ is linearly equivalent to the \emph{double point divisor from inner projection}
$$
D_{\inn}:=-K_X + (d-n-e-1)H
$$
and $D_{\inn}$ is semiample except when $X \subseteq \nP^r$ is a Roth variety, a scroll over a smooth projective curve, or the second Veronese surface, which we call a \emph{Noma's exceptional variety}. 
More precisely, the base locus of the linear subsystem $\lvert D_{\inn} \rvert$ spanned by geometric divisors $D_{\inn}(\pi_{\Lambda, X})$ for general inner projections $\pi_{\Lambda, X} \colon X \dashrightarrow \overline{X} \subseteq \nP^{n+1}$ is the exactly same to $\sC(X)$, so in particular, $\Bs \lvert D_{\inn} \rvert \subseteq \sC(X)$ (\cite[Theorem 1]{N1}). As $\dim \sC(X) \leq 1$ and the equality holds if and only if $X \subseteq \nP^r$ is a Roth variety but not a rational scroll (\cite[Theorem 2]{N1}), $\Bs \lvert D_{\inn} \rvert$ is at most a finite set and thus $D_{\inn}$ is semiample when $X \subseteq \nP^r$ is not a Noma's exceptional variety.

\medskip

It is natural to ask when $D_{\inn}$ is base point free. This problem only makes sense if $\sC(X)$ is a nonempty finite set. For $x \in \sC(X)$, let $Y_x$ be the image of the inner projection $\pi_{x, X} \colon X \dashrightarrow \nP^{r-1}$ centered at $x$. Noma proved that if $Y$ is a rational normal scroll, then $D_{\inn}$ is base point free unless $Y = Q^2 \subseteq \nP^3$ is a quadric surface (\cite[Theorem 0.3]{N3}). See Example \ref{ex:Nomaexcsurf} for more details on the exceptional case. In this paper, we show that if $Y_x$ is smooth for every $x \in \sC(X)$, then $D_{\inn}$ is base point free unless $Y_x = Q^2 \subseteq \nP^3$ for some $x \in \sC(X)$.

\begin{theorem}\label{thm:bpf}
Let $X \subseteq \nP^r$ be a non-degenerate smooth projective variety of dimension $n \geq 2$, codimension $e \geq 2$ such that it is not a Noma's exceptional variety. Suppose that $\sC(X) \neq \emptyset$ and $Y:=\overline{\pi_{x,X}(X \setminus \{x\})}$ is smooth for a point $x \in \sC(X)$. Then $x \not\in \Bs \lvert D_{\inn} \rvert$ except when $Y=Q^2 \subseteq \nP^3$ is a quadric surface (cf. Example \ref{ex:Nomaexcsurf}).
\end{theorem}

In the situation of Theorem \ref{thm:bpf}, Noma's result \cite[Theorem 4]{N2} says that the blow-up $\widetilde{X}$ of $X$ at $x$ with the exceptional divisor $E$ is a smooth member of $\lvert \sO_F(\mu) \otimes \pi^* L \rvert$ for $\mu=\deg \pi_{x,X}$, where $F:=\nP (\sO_Y \oplus \sO_Y(1))$ with the canonical projection $\tau_F \colon F \to Y$ for $L= \sO_Y(\tau_F(E)) \in \Pic(Y)$. Note that $x \not\in \Bs \lvert D_{\inn} \rvert$ if and only if the restriction map $H^0(\widetilde{X}, \sO_{\widetilde{X}}(\sigma^*D_{\inn})) \to H^0(E, \sO_E)$ is surjective. We show the surjectivity of the restriction map through cohomological arguments using the geometry of $F$ and $Y$. If one sets up the problem well, then it can be easily solved when $Y \subseteq \nP^{r-1}$ is not Noma's exceptional variety. However, in the case of Noma's exceptional variety, a careful case-by-case analysis is needed.
We are not aware of any example where $D_{\inn}$ is not base point free other than Noma's example where $Y=Q^2$ (Example \ref{ex:Nomaexcsurf}). It would be exceedingly interesting to study what happens when $Y$ is singular. 

\medskip

Recall from \cite[Corollary 5.3]{Il} that $D_{\out}$ is big if and only if $X \subseteq \nP^r$ is not the Segre embedding of $\nP^1 \times \nP^{n-1}$. It is also worth to study when $D_{\inn}$ is big. Here we keep assuming that $X \subseteq \nP^r$ is not a Noma's exceptional variety. If $D_{\inn}$ is not big, then $D_{\inn}$ is base point free (Proposition \ref{prop:baselocus}). We expect that there are only a few cases where $D_{\inn}$ is not big, and in each case, the morphism given by $\lvert D_{\inn}\rvert$ would have a special structure like the adjunction mapping. We have three examples where $D_{\inn}$ is not big (Example \ref{ex:nonbig}): unless $X \subseteq \nP^r$ is a del Pezzo variety (equivalent to $D_{\inn}=0$ by \cite[Lemma 2.5]{KP2}), the morphism given by $\lvert D_{\inn}\rvert$ is either a conic fibration over $\nP^{n-1}$ or a linear fibration over $Q^{n-1}$. On the other hand, if $X \subseteq \nP^r$ is not a del Pezzo variety, $D_{\inn}$ is big in the following cases: 
\begin{enumerate}
\item $-K_X$ and $H$ are proportional. This holds whenever $e \leq n-2$ by Barth--Larsen theorem (cf. \cite[Corollary 3.2.3]{positivity}).
\item $X \subseteq \nP^r$ is not linearly normal. Indeed, $-K_X+(d-n-e-2)H$ is semiample, so $D_{\inn}=-K_X+(d-n-e-2)H+H$ is ample.
\end{enumerate}
Our result in this direction is the following.

\begin{theorem}\label{thm:D_innbig}
Let $X \subseteq \nP^r$ be a non-degenerate smooth projective variety of dimension $n \geq 2$, codimension $e \geq n+1$, and degree $d$. Assume that $X \subseteq \nP^r$ is neither a Noma's exceptional variety nor a del Pezzo variety. If
$$
d \geq \left\lceil \frac{2n-e-2}{2e-2n} \right\rceil e+1,
$$
then $D_{\inn}$ is big except when $e=n+1, d=2n$ and $\lvert D_{\inn} \rvert$ induces a linear fibration $X \to Q^{n-1}$ (cf. Example \ref{ex:nonbig}).
\end{theorem}

If $e \geq (4n-2)/3$, then the degree condition in Theorem \ref{thm:D_innbig} becomes $d \geq e+1$, which is automatically satisfied (see Corollary \ref{cor:e>=(4n-2)/3=>D_innbig}). In other words, $D_{\inn}$ becomes big as soon as $e$ is large enough. Conversely, when $e \geq n+1$ and $D_{\inn}$ is not big, we have $d \leq (4n+1)/3$, so $X$ belongs to a bounded family once we fix the dimension $n$. It would be possible to classify all such cases. On the other hand, we also derive the following consequence asserting that the Iitaka dimension $\kappa(D_{\inn})$ is large. This suggests that the image of the morhpism given by $\lvert D_{\inn} \rvert$ might be big in contrast to the adjunction mapping. 

\begin{corollary}\label{cor:Iitakadim}
Let $X \subseteq \nP^r$ be a non-degenerate smooth projective variety of dimension $n \geq 2$, codimension $e \geq 2$, and degree $d$. Assume that $X \subseteq \nP^r$ is neither a Noma's exceptional variety nor a del Pezzo variety. Then $\kappa(D_{\inn}) \geq (3n-1)/4$. 
\end{corollary}

To prove Theorem \ref{thm:D_innbig}, it is enough to show $D_{\inn}^n>0$ because $D_{\inn}$ is nef, but it is not easy to control the intersection numbers $(-K_X)^i.H^{n-i}$. Instead, we utilize the technique developed by the second author with Sijong Kwak in \cite{KP1, KP2} based on Siu's bigness criterion (cf. \cite[Theorem 2.2.15]{positivity}): if $D,E$ are nef divisors and $nD^{n-1}.E < D^n$, then $D-E$ is big. We take $D:=(d-e-2)H$ and $E:=K_X+(n-1)H$. Note that $D-E=E_{\inn}$ and $nD^{n-1}.E < D^n$ is equivalent to $2g < ((d-e-2)d+2n)/n$, where $g$ is the sectional genus of $X \subseteq \nP^r$. We then apply Castelnuovo's genus bound (see Subsection \ref{subsec:Castgenusbd}). This method works well when $e \geq n+1$. If $e \leq n-2$, then Barth--Larsen theorem implies that $D_{\inn}$ is big. It remains to figure out when $D_{\inn}$ is not big for the case of $e=n-1$ or $n$. For small dimensional cases, this problem can be approached using specific geometric properties. In particular, we show that $D_{\inn}$ is big for surfaces and threefolds except one case.

\begin{theorem}\label{thm:D_innbigdim=3}
Let $X \subseteq \nP^r$ be a non-degenerate smooth projective variety of dimension $n =2$ or $3$, codimension $e \geq 2$, and degree $d$. Assume that $X \subseteq \nP^r$ is neither a Noma's exceptional variety nor a del Pezzo variety. Then $D_{\inn}$ is big except when $n=3, e=3, d=6$ and $\lvert D_{\inn} \rvert$ induces a conic fibration $X \to \nP^2$ (cf. Example \ref{ex:nonbig}).
\end{theorem}

One motivation for studying the positivity of double point divisors is that it has applications to the Castelnuovo--Mumford regularity. A linear bound for $\operatorname{reg}(X)$ was shown using the property that $D_{\out}$ is base point free (\cite[Theorem 3.12]{BM}, \cite[Theorem C]{KP2}), and a sharp bound for $\operatorname{reg}(\sO_X)$ was proven using the property that $D_{\inn}$ is semiample (\cite[Theorem B]{KP2}, \cite[Corollary 5]{N1}). This gives evidence for the Eisenbud--Goto regularity conjecture for smooth projective varieties. We refer to \cite{KP2} for more details. We hope that the stronger positivity of double point divisors obtained in this paper would lead to a better regularity bound.

\medskip

The paper is organized as follows. In Section \ref{sec:prelim}, we introduce basic notions and collect some known facts. Sections \ref{sec:va1} and \ref{sec:va2} are devoted to the proof of Theorem \ref{thm:D_outva}. We show Theorem \ref{thm:bpf} in Section \ref{sec:bpf} and Theorems \ref{thm:D_innbig} and \ref{thm:D_innbigdim=3} and Corollary \ref{cor:Iitakadim} in Section \ref{sec:big}. 

\medskip

\noindent {\bf Acknowledgments.} We are grateful to Sijong Kwak and Atsushi Noma for valuable discussion and encouragement.

%%%%%%%%%%%%%%%%%%%%%%%%%%%%%%%%%%%%%%%%%%%%%%%%%%%%%%%
\section{Preliminaries}\label{sec:prelim}
%%%%%%%%%%%%%%%%%%%%%%%%%%%%%%%%%%%%%%%%%%%%%%%%%%%%%%%

\subsection{Roth varieties}
A \emph{Roth variety} $X \subseteq \nP^r$ is an isomorphic projection of a divisor on a rational normal scroll $\overline{S} \subseteq \nP^N$ which is the image of the birational embedding $\psi$ of $S:=\nP(\sO_{\nP^1}^{\oplus 2} \oplus \sO_{\nP^1}(a_1) \oplus \cdots \oplus \sO_{\nP^1}(a_{n-1}))$ with $1 \leq a_1 \leq \cdots a_{n-1}$ given by $\lvert \sO_S(1) \rvert$ such that $X$ is a general smooth member of $\lvert \sO_S(b) \otimes \pi^* \sO_{\nP^1}(1)\rvert$ for an integer $b \geq 1$, where $\pi \colon S \to \nP^1$ is the canonical projection. Here $N=a_1+\cdots+a_{n-1}+n$ and $\deg \overline{S}=a_1+ \cdots + a_{n-1}$. The birational morphism $\psi \colon S \to \overline{S}$ contracts $\Gamma:=\nP(\sO_{\nP^1}^{\oplus 2}) \subseteq S$ to a line $\ell \subseteq \overline{S} \subseteq \nP^N$. Note that $\psi|_X \colon X \to \nP^N$ is an embedding and $\ell = \psi(\Gamma) = \psi(X \cap \Gamma) \subseteq X$. Here $X \subseteq \nP^N$ is non-degenerate and linearly normal, and $\deg X = b(a_1 + \cdots + a_{n-1})+1$. If $b=1$, then the Roth variety $X \subseteq \nP^r$ is a rational scroll. See \cite[Section 3]{Il} for more details.

\subsection{Segre loci}
In this subsection, we do not assume that $X$ is smooth. The \emph{outer Segre locus} of $X \subseteq \nP^r$ is defined as
$$
\sB(X):=\{u \in \nP^r \setminus X \mid \length(X \cap \langle x, u \rangle ) \geq 2~\text{ for general $x \in X$} \}.
$$
Note that $u \in \sB(X)$ if and only if the outer projection $\pi_{u, X} \colon X \to \nP^{r-1}$ centered at $u$ is nonbirational. Keep in mind that we always assume $e \geq 2$. Then we have $\dim \sB(X) \leq \min\{n-1, \dim \Sing X+1\}$ (\cite[Theorem 4]{N0}). When $X$ is smooth, we put $\dim \Sing X = -1$. The \emph{inner Segre locus} of $X \subseteq \nP^r$ is defined as
$$
\sC(X):=\{u \in X \mid \length(X \cap \langle u, x \rangle) \geq 3~\text{for general $x \in X$} \}.
$$
Note that $u \in \sC(X)$ if and only if the inner projection $\pi_{u, X} \colon X \dashrightarrow \nP^{r-1}$ centered at $u$ is nonbirational. As $e \geq 2$, we have $\dim \sC(X) \leq \min\{ n-1, \dim \Sing X + 2\}$ (\cite[Theorem 5]{N0}). When $X$ is smooth, $\dim \sC(X) \leq 1$ and the equality holds if and only if $X \subseteq \nP^r$ is a Roth variety but not a rational scroll. We refer to \cite{N0, N1, N2} for more details.

\subsection{Birational divisors on conical scrolls}\label{subsec:birdiv}
We assume again that $X$ is smooth. If $\sC(X)$ is a nonempty finite set and $x \in \sC(X)$, then the inner projection $\pi_{x, X} \colon X \dashrightarrow \overline{Y} \subseteq \nP^{r-1}$ centered at $x$ is not birational. Note that the cone
$$
\Cone(x,X):=\overline{\bigcup_{u \in X \setminus \{x\}} \langle x,u \rangle}
$$
is a conical scroll over $\overline{Y}$. There is a non-degenerate birational embedding $\nu \colon Y \to \overline{Y} \subseteq \nP^{r-1}$ from a smooth projective variety $Y$ of dimension $n$ such that for $F:=\nP (\sO_Y \oplus \sO_Y(1))$ with $\sO_Y(1)=\nu^*\sO_{\overline{Y}}(1)$, a linear subsystem of $\lvert \sO_F(1) \rvert$ gives a birational embedding $\sigma_F \colon F \to \nP^r$ with $\overline{F}:=\sigma_F(F)=\Cone(x,X)$ contracting $\Gamma := \nP(\sO_Y) \subseteq F$ to $x \in \overline{F}$. If $\tau_F \colon F \to Y$ is the canonical projection, then $\sO_F(\Gamma) = \sO_F(1) \otimes \tau_F^* \sO_Y(-1)$. By \cite[Theorem 4]{N2}, $X$ is a \emph{birational divisor of type $(\mu, L)$} on a conical scroll $\overline{F}$: there is a prime divisor $\widetilde{X} \in \lvert \sO_F(\mu) \otimes \tau_F^* L \rvert$ for some $\mu \geq 2$ and $L \in \Pic(Y)$ with $H^0(Y, L) \neq 0$ and  $L.\sO_Y(1)^{n-1} = 1$ such that $\sigma:=\sigma_F|_{\widetilde{X}} \colon \widetilde{X} \to X \subseteq \nP^r$ is a non-degenerate birational embedding and $\mu=\deg \pi_{x,X}, d = \mu \deg \overline{Y} + 1$. If $Y$ is smooth, then we may take $\nu=\operatorname{id}_Y$. In this case, $\sigma$ is the blow-up of $X$ at $x$, and $E:=\Gamma \cap \widetilde{X}$ is the exceptional divisor. Using \cite[Theorem 5]{N2} and Bertini theorem, one can construct smooth birational divisors on conical scrolls (see \cite[Example 4.2]{N2} and Example \ref{ex:singularY}).

\subsection{Double point divisors from outer projection}
Let $\Lambda$ be an $(e-2)$-dimensional linear subspace of $\nP^r$ with $X \cap \Lambda = \emptyset$, and consider the outer projection
$$
\pi_{\Lambda, X} \colon X \longrightarrow \overline{X} \subseteq \nP^{n+1}
$$
centered at $\Lambda$. Notice that $\pi_{\Lambda, X}$ is isomorphic in a neighborhood of $x \in X$ if and only if $\Lambda \cap \Cone(x, X) = \emptyset$. Thus if $\Lambda \cap \Cone(x, X) = \emptyset$ for some $x \in X$, then $\pi_{\Lambda, X}$ is finite and birational. Suppose that $\pi_{\Lambda, X}$ is finite and birational. Note that $\deg \overline{X} = d$. By birational double--point formula (cf. \cite[Lemma 10.2.8]{positivity}), there exists an effective divisor $D_{\out}(\pi_{\Lambda, X})$ on $X$ such that
$$
D_{\out}(\pi_{\Lambda, X}) \sim -K_X + \pi_{\Lambda, X}^*K_{\overline{X}} = -K_X + (d-n-2)H=D_{\out}.
$$
Note that $\Supp(D_{\out}(\pi_{\Lambda, X}))$ is the non-isomorphic locus of $\pi_{\Lambda, X}$. If $S_{\Lambda,X}$ is the cone over $\pi_{\Lambda,X}(D_{\out}(\pi_{\Lambda,X}))$ with vertex at $\Lambda$, then scheme-theoretically 
$D_{\out}(\pi_{\Lambda, X}) = X \cap S_{\Lambda, X}$.
Recall that the linear subsystem $\lvert V_{\out} \rvert$ generated by geometric divisors is base point free (\cite[Technical Appendix 4]{BM}). Furthermore, $\lvert V_{\out} \rvert$ separates two distinct points and hence $D_{\out}$ is ample  unless $X \subseteq \nP^r$ is a Roth variety (\cite[Theorem 4.2]{Il}). When $X \subseteq \nP^r$ is a Roth variety, $D_{\out}.\ell = 0$ for the line $\ell=\psi(\Gamma) \subseteq X$ and thus $D_{\out}$ is not ample (\cite[Proposition 3.8]{Il}).

\subsection{Double point divisors from inner projection}
Let $x_1, \ldots, x_{e-1} \in X$ be general points, and consider the inner projection 
$$
\pi_{\Lambda, X}\colon X\dashrightarrow \overline{X} \subseteq \nP^{n+1}
$$
centered at $\Lambda:= \langle x_1, \ldots, x_{e-1} \rangle \subseteq \nP^r$. By the general position lemma (cf. \cite[Lemma 1.2]{N1}), $\Lambda$ is an $(e-2)$-dimensional linear subspace in $\nP^r$. 
Let $\sigma \colon \widetilde{X} \to X$ be the blow-up of $X$ at $x_1, \ldots, x_{e-1}$. There exists a birational morphism $\widetilde{\pi}_{\Lambda, X} \colon \widetilde{X} \to \overline{X}$. We have the following commutative diagram
\[\xymatrix{
\widetilde{X} \ar^{\sigma}[r] \ar_{\widetilde{\pi}_{\Lambda, X}}[rd] & X \ar@{.>}[d]^-{\pi_{\Lambda, X}} \ar@{^{(}->}[r] & \nP^{r} \ar@{.>}[d]\\
& \overline{X} \ar@{^{(}->}[r] & \nP^{n+1}.
}\]
Noma \cite[Theorem 3]{N1} proved that if $X \subseteq \nP^r$ is neither a scroll over a smooth projective curve nor the second Veronese surface, then $\widetilde{\pi}_{\Lambda, X} \colon \widetilde{X} \to X$ does not contract any divisor. In this case, $\deg(\overline{X})=d-e+1$. By the birational double--point formula, there exists an effective divisor $\widetilde{D}(\widetilde{\pi}_{\Lambda, X})$  on $\widetilde{X}$ such that
$$
\widetilde{D}(\widetilde{\pi}_{\Lambda, X}) \sim -K_{\widetilde{X}} + \widetilde{\pi}_{\Lambda, X}^*K_{\overline{X}}.
$$
Then 
$$
D_{\inn}(\pi_{\Lambda, X}) = \sigma_*\widetilde{D}(\widetilde{\pi}_{\Lambda, X}) \sim -K_X + (d-n-e-1)H = D_{\inn}.
$$
Note that $\Supp (D_{\inn}(\pi_{\Lambda, X}))$ is the divisorial part of the closure of the non-isomorphic locus of $\pi_{\Lambda, X}$. Recall that $\Bs \lvert D_{\inn} \rvert \subseteq \sC(X)$ (\cite[Theorem 1]{N1}). When $X \subseteq \nP^r$ is not a Roth variety, $\sC(X)$ is at most a finite set. Thus if $X \subseteq \nP^r$ is not a Noma's exceptional variety (a Roth variety, a scroll over a smooth projective curve, or the second Veronese surface), then Fujita--Zariski theorem (cf. \cite[Remark 2.1.32]{positivity}) implies that $D_{\inn}$ is semiample (\cite[Theorem 4]{N1}).

\subsection{Castelnuovo's genus bound}\label{subsec:Castgenusbd}
Let $C \subseteq \nP^r$ be a non-degenerate smooth projective curve of genus $g$, codimension $e$, and degree $d$. Castelnuovo's genus bound (\cite[page 116]{ACGH}) says that
\begin{equation}\label{eq:Castelnuovo-big}
g \leq \frac{(d-1-\epsilon)(d+\epsilon-e-1)}{2e} = \frac{(d-e-2)d+(\varepsilon+1)(e+1-\varepsilon)}{2e},
\end{equation}
where $m:=\lfloor (d-1)/e \rfloor$ and $\varepsilon :=(d-1)-me$. Note that $m \geq 1$ and $0 \leq \varepsilon \leq e-1$. When the equality holds, $C \subseteq \nP^r$ is projectively normal (\cite[page 117]{ACGH}). If furthermore $d \geq 2e+3$, then $C$ is a smooth divisor on a rational normal scroll except when $e=4$ and $d=2k$, in which case, $C \subseteq \nP^5$ is the image of the plane curve of degree $k$ under the second Veronese embedding $v_2 \colon \nP^2 \to \nP^5$ (\cite[Theorem 2.5]{ACGH}).

%%%%%%%%%%%%%%%%%%%%%%%%%%%%%%%%%%%%%%%%%%%%%%%%%%%%%%%
\section{Very ampleness of $S_{\out}$}\label{sec:va1}
%%%%%%%%%%%%%%%%%%%%%%%%%%%%%%%%%%%%%%%%%%%%%%%%%%%%%%%

This section is devoted to the proof of Theorem \ref{thm:D_outva} $(1)$. We first show some lemmas, which are certainly well-known, but we write down the detailed proofs for the reader's convenience.

\begin{lemma}\label{lem: smoothness of hyperplane sections}
Let $L$ be a line in $\nP^r$ passing through a point $x \in X$. If $n \geq 3$ and $H_1,\ldots,H_{n-2} \subseteq \nP^r$ are general hyperplanes containing $L$, then  $X_i := X \cap H_1 \cap \cdots \cap H_i$ is smooth for each $i=1,\ldots,n-2$.
\end{lemma}

\begin{proof}
By Bertini theorem, we may assume $X_i$ is smooth at the points in $X_i \setminus L$. We claim by induction on $i$ that $X_i$ is smooth at the points of $L \cap X_i$. For $i=0$, clearly $X_0 := X$ is smooth. For $i \leq n-3$, if $X_i$ is smooth at the points of $L \cap X_i$, then for each $y \in L \cap X_i$, the projective tangent space $\mathbf{T}_y X_i$ has codimension $r-n$ in $\nP^{r-i} = H_1\cap \cdots \cap H_{i}$. 
The hyperplanes containing $\mathbf{T}_yX_i$ form a linear subspace of dimension $r-n-1$ in $(\nP^{r-i})^\vee$. Since $\dim (L \cap X_i) \leq 1$, the locus of hyperplanes containing at least one of $\{\mathbf{T}_y X_i\}_{y \in L \cap X_i}$ is of dimension at most $r-n$ in $(\nP^{r-i})^\vee$. On the other hand, the hyperplanes containing $L$ form a linear subspace of dimension $r-i-2 \geq r-n+1$. Hence we may pick $H_{i+1}$ which does not contain any of $\mathbf{T}_y X_i$. Then $X_{i+1} = X_i \cap H_{i+1}$ is smooth at the points of $X_{i+1} \cap L$.
%
%In this proof we argue in general when a hyperplane section is smooth at a given point. Let Yn⊂\nPNY^n \subseteq \nP^N be a smooth projective variety, and let H⊂\nPNH \subseteq \nP^N be a hyperplane defined by a linear form ℓ=a0Z0+…+aNZN∈\nC[Z0,…,ZN]\ell = a_0Z_0 + \ldots + a_N Z_N \in \nC[Z_0,\ldots,Z_N]. The Jacobian matrix JYJ_Y has rank N−nN-n at every point of YY. Let V⊂\nCN+1V \subseteq \nC^{N+1} be the row space of JYJ_Y at a point y∈Y∩Hy \in Y \cap H. Then, it is immediate from the definition of smoothness that Y∩HY \cap H is smooth at yy if and only if (a0,…,aN)∉V(a_0,\ldots,a_N) \not\in V.
%In our situation, Xn−ii⊂\nPn−i+2X_i^{n-i} \subseteq \nP^{n-i+2} is a variety of codimension 22. The Jacobian matrix JXiJ_{X_i} has rank 22 at the smooth point xx, and the row space Vi⊂\nCn−i+3V_i \subseteq \nC^{n-i+3} has dimension 22. Since we are finding a hyperplane Hi+1⊂\nPn−i+2H_{i+1} \subseteq \nP^{n-i+2} containing ℓv\ell_v, there is a (n−i+1)(n-i+1)-dimensional choices for ℓ\ell. Thus, we are able to choose the hyperplane Hi+1H_{i+1} such that Xi+1:=Xi∩Hi+1X_{i+1} := X_i \cap H_{i+1} is smooth at xx as soon as n−i+1>2n-i+1 > 2, or equivalently, $i < n-1$.
\end{proof}

By the above lemma, we can always take a smooth surface section $S \subseteq \nP^{r-n+2}$ of $X \subseteq \nP^r$ such that $L$ is still contained in $\nP^{r-n+2}$. To attain a curve section $C \subseteq \nP^{r-n+1}$ such that $L$ is contained in $\nP^{r-n+1}$, we should sacrifice the smoothness when $L \subseteq S$. 

\begin{corollary}\label{cor: Vout separates tangents, reduction outcome}
Let $L$ be a line in $\nP^r$ passing through a point $x \in X$, and $C:=X \cap H_1 \cap \cdots \cap H_{n-1}$ for general hyperplanes $H_1, \ldots, H_{n-1} \subseteq \nP^r$ containing $L$.
\begin{enumerate}
    \item If $L \not\subseteq X$, then $C$ is smooth.
    \item If $L \subseteq X$, then $C$ is the union of $L$ and a smooth curve $C'$ with $x \not\in C'$.
\end{enumerate}
\end{corollary}

\begin{proof}
By Lemma \ref{lem: smoothness of hyperplane sections}, $S := X \cap H_1 \cap \cdots \cap H_{n-2}$ is a smooth projective surface. To prove (1), assume $L \not \subseteq S$. The hyperplanes in $\nP^{r-n+2} = H_1\cap \cdots \cap H_{n-2}$ containing at least one of $\{ \nT_y S \}_{y \in L \cap S}$ form a family of dimension at most $r-n-1$, while the hyperplanes in $\nP^{r-n+2}$ containing $L$ form a linear space of dimension $r-n$. Thus there exists a hyperplane $H_{n-1}$ not containing any of $\nT_y S$. This proves (1).
To prove (2), assume $L \subseteq S$. The image of the restriction map $H^0(\nP^{r-n+2}, \sI_{L|\nP^{r-n+2}}(1)) \rightarrow H^0(S, \sI_{L|S}(1))$ gives rise to a base point free linear subsystem of $\lvert \sO_{S}(1) \otimes \sO_{S}(-L) \rvert$. Then (2) follows from Bertini theorem.
\end{proof}

\begin{lemma}\label{lem:strangetangent}
Let $L$ be a projective tangent line to $X$ at a point $x \in X$. Then $\nT_y X \cap L = \emptyset$ for a general point $y \in X$.
\end{lemma}

\begin{proof}
To derive a contradiction, suppose that $\nT_y X \cap L \neq \emptyset$ for every $y \in X$. Consider a general curve section $C \subseteq \nP^{e+1}$ of $X \subseteq \nP^r$ such that $\nT_x C = L$. By Corollary \ref{cor: Vout separates tangents, reduction outcome}, $C$ is smooth when $L \not\subseteq X$ or $C=L \cup C'$ such that $C'$ is smooth with $x \not\in C'$ when $L \subseteq X$. Consider the projection $\pi_{p,C} \colon C \dashrightarrow \overline{C} \subseteq \nP^e$ centered at a general point $p \in L$. Notice that all projective tangent lines to $\overline{C}$ pass through one point, i.e, $\overline{C} \subseteq \nP^e$ is a strange curve. However, there is no strange curve other than a line in characteristic zero, so we get a contradiction.
\end{proof}

We are ready to prove Theorem \ref{thm:D_outva} $(1)$. 
\begin{comment}
For an $(e-2)$-dimensional linear subspace $\Lambda$ of $\nP^r$ with $\Lambda \cap X = \emptyset$, let $S_{\Lambda, X}$ be the union of $(e-1)$-dimensional linear subspace $M$ of $\nP^r$ containing $\Lambda$ with $\length(X \cap M) \geq 2$. Notice that $S_{\Lambda,X}$ is the cone over $\pi_{\Lambda,X}(D_{\out}(\pi_{\Lambda,X}))$ with vertex at $\Lambda$ and scheme-theoretically 
$$
D_{\out}(\pi_{\Lambda, X}) = X \cap S_{\Lambda, X}.
$$
In particular, we have
$$
\nT_x D_{\out}(\pi_{\Lambda, X}) = \nT_x X \cap \nT_x S_{\Lambda, X} ~~\text{ for any $x \in \Supp (D_{\out}(\pi_{\Lambda, X}))$}.
$$
\end{comment}

\begin{proof}[Proof of Theorem \ref{thm:D_outva} $(1)$]
First, consider the case when $\sC(X) \neq \emptyset$ and $x \in \sC(X)$. Let 
$\pi_{\Lambda, X} \colon X \rightarrow \overline{X} \subseteq \nP^{n+1}$ be an outer projection centered at an $(e-2)$-dimensional linear subspace $\Lambda$ in $\nP^r$ such that  $D_{\out}(\pi_{\Lambda, X})$ is defined and $x \in \Supp(D_{\out}(\pi_{\Lambda, X}))$.  
We claim that 
$$
\mult_x D_{\out}(\pi_{\Lambda, X}) \geq 2,
$$
which implies that $S_{\out}$ does not separate any tangent directions at $x$. For a general $(e+1)$-dimensional linear subspace $M$ in $\nP^r$ containing $\Lambda$ and $x$, consider the curve section
$$
C:=X \cap M \subseteq M = \nP^{e+1}
$$ 
of $X \subseteq \nP^r$. By Bertini theorem, $C$ is smooth. Note that 
$$
D_{\out}(\pi_{\Lambda, X})|_C = D_{\out}(\pi_{\Lambda, C})~~\text{and}~~\mult_x D_{\out}(\pi_{\Lambda, X}) = \mult_x D_{\out}(\pi_{\Lambda, C}).
$$
Since $\pi_{\Lambda,C}^{-1}(\pi_{\Lambda, C}(x)) = C \cap \langle \Lambda, x \rangle$ and $\length( C \cap \langle \Lambda, x \rangle) \geq 3$,  it follows that  $\mult_{\pi_{\Lambda, C}(x)} \overline{C} \geq 3$.
Now, take a log resolution $f \colon S \to \nP^2$ of $(\nP^2, \overline{C})$, and write $f^*\overline{C}=C + F$ for some $f$-exceptional divisor $F$. Then
$$
-(K_{S/\nP^2}-F)|_C = -\big((K_S+C)-f^*(K_{\nP^2}+\overline{C})\big)|_C = -K_C + f^* K_{\overline{C}}=D_{\out}(\pi_{\Lambda, C}).
$$ 
As $\mult_{\pi_{\Lambda, C}(x)} \overline{C} \geq 3$, we obtain $\mult_x D_{\out}(\pi_{\Lambda,C})  \geq 2$.

\medskip

Next, consider the case when $x \in X \setminus \sC(X)$. For any projective tangent line $L$ to $X$ at $x$, we show that there is a section in $S_{\out}$ separating the tangent direction at $x$ corresponding to $L$. By Lemma \ref{lem:strangetangent}, $\nT_y X \cap L = \emptyset$ for a general point $y \in X$. As $x \not\in \sC(X)$, we may assume that $X \cap \langle x,y \rangle = \{x,y\}$ as a scheme. Let $H$ be a general hyperplane in $\nP^r$ such that $\nT_y X \subseteq H, ~x \in H$, and $L \not\subseteq H$. We can take a general $(e-2)$-dimensional linear subspace $\Lambda$ in $H$ such that $\Lambda \cap \nT_y X = \emptyset$, $\Lambda \cap \langle x,y \rangle = \Lambda \cap \Cone(x,X)= \{p\}$
as a scheme for a point $p \not\in X$, and $\Lambda \cap \Cone(q, X) = \emptyset$ for a general point $q \in X$. In particular, $\Lambda \cap X = \emptyset$. Note that $\Lambda \cap \nT_x X \subseteq \Lambda \cap \Cone(x, X)= \{ p \}$ and $p \in \langle x,y \rangle$. If $\Lambda \cap \nT_x X = \{ p \}$, then $\langle x,p \rangle = \langle x,y \rangle \subseteq \nT_x X$, which is a contradiction. Thus $\Lambda \cap \nT_x X = \emptyset$. Suppose that there is a point $u \in \langle \Lambda, x \rangle \cap X$ with $u \neq x,y$. Since $X \cap \langle x,y \rangle = \{x,y\}$, it follows that $u \not\in \langle x,y \rangle$. Then $\langle x,u \rangle$ intersects $\Lambda$ at some point $v \neq p$ so that $v \in \Lambda \cap \Cone(x,X)$, which is a contradiction to that $\Lambda \cap \Cone(x,X) = \{ p \}$. Thus $\langle \Lambda, x \rangle \cap X = \{x,y\}$ as a set. However, $\Lambda \cap \nT_x X = \emptyset$ and $\Lambda \cap \nT_y X = \emptyset$, so $\langle \Lambda, x \rangle \cap X = \{x,y\}$ as a scheme. Now, as $\Lambda \cap \Cone(p,X)=\emptyset$ for a general point $p \in X$, the outer projection 
$$
\pi_{\Lambda, X} \colon X \longrightarrow \overline{X} \subseteq \nP^{n+1}
$$
centered at $\Lambda$ is finite birational and hence $D_{\out}(\pi_{\Lambda,X})$ is defined. We have
$$
\pi_{\Lambda, X}^{-1}(\pi_{\Lambda, X}(x)) = \langle \Lambda, x \rangle \cap X = \{ x,y\}.
$$
Notice that $L \not\subseteq \langle \Lambda, \nT_y X \rangle$ since $\langle \Lambda, \nT_y X \rangle \subseteq H$ but $L \not\subseteq H$. Thus $\pi_{\Lambda, X}(L) \not\subseteq \pi_{\Lambda, X}(\nT_y X)$. As $\pi_{\Lambda, X}(L) \subseteq \pi_{\Lambda, X}(\nT_x X)$, we see that $\pi_{\Lambda, X}(\nT_x X)$ and $\pi_{\Lambda, X}(\nT_y X)$ meet transversally. Note that 
$$
\pi_{\Lambda, X}(L) \not\subseteq 
 \pi_{\Lambda, X}(\nT_x X) \cap \pi_{\Lambda, X}(\nT_y X) = \nT_{\pi_{\Lambda, X}(x)}\pi_{\Lambda, X}(D_{\out}(\pi_{\Lambda, X})).
$$
Hence  $L \not\subseteq \nT_x D_{\out}(\pi_{\Lambda, X})$. This means the section of $D_{\out}(\pi_{\Lambda, X})$ separates the tangent direction at $x$ corresponding to $L$.
\end{proof}

\begin{remark}\label{rem:Ilic'sproof}
Suppose that $X \subseteq \nP^r$ is a smooth projective curve.
In this case, it was claimed in \cite[Proof of Proposition 2.7]{Il} that $S_{\out}$ separates tangent vectors at each point $x \in X$. This contradicts Theorem \ref{thm:D_outva} $(1)$. In fact, the proof in \cite{Il} overlooks the possibility that $\sC(X)$ could be nonempty (see Example \ref{ex:curvecase} $(1)$ for explicit examples of $\sC(X) \neq \emptyset$). However, the statement of \cite[Proposition 2.7]{Il} that $D_{\out}$ is very ample is quite correct because one can easily check that $\deg D_{\out} \geq 2g+1$ using Castelnuovo's genus bound (\ref{eq:Castelnuovo-big}) as was pointed out in \cite[page 1648]{Il}. Here we want to stress that the tangent separation problem of $V_{\out} $ rather than $D_{\out}$ is not easy at all even for the curve case. Most part of the next section is devoted to proving that $\lvert V_{\out} \rvert$ separates tangent vectors in the curve case. 
\end{remark}

\begin{comment}
\begin{remark}
Let $C \subseteq \nP^r$ be a non-degenerate smooth projective curve of genus $g$ and degree $d$. Castelnuovo's genus bound says that
$$
g \leq \frac{(d-1-\epsilon)(d+\epsilon-e-1)}{2e}~~\text{where $\epsilon = (d-1)-\left\lfloor \frac{d-1}{e} \right\rfloor e$}.
$$
It is easy to verify that $d^2-3d+1 \geq (d-1-\epsilon)(d+\epsilon-e-1)$, which implies that
$$
\frac{(d-3)d+1}{4} \geq \frac{(d-1-\epsilon)(d+\epsilon-e-1)}{4} \geq \frac{(d-1-\epsilon)(d+\epsilon-e-1)}{2e} \geq g.
$$
Thus we have
$$
\deg D_{\out} = (d-3)d - 2g+2 \geq 2g+1,
$$
so $D_{\out}$ is very ample.
\end{remark}
\end{comment}

%%%%%%%%%%%%%%%%%%%%%%%%%%%%%%%%%%%%%%%%%%%%%%%%%%%%%%%
\section{Very ampleness of $\lvert V_{\out} \rvert$}\label{sec:va2}
%%%%%%%%%%%%%%%%%%%%%%%%%%%%%%%%%%%%%%%%%%%%%%%%%%%%%%%

\subsection{Overall framework}\label{subsec: 4.1}
%In this note\cmt{!} we give a generalization (and a correction\cmt{!}) of \cite[Proposition~2.7]{Il}, which essentially states that $\lvert V_{\out} \rvert$ is very ample in the curve case. More precisely, we give a proof of the following.
% \begin{theorem}\label{thm: Vout separates tangents}
% Let $X^n \subseteq \nP^N\,(N \geq n+2)$ be a nondegenerate smooth projective variety. Assume that the set
% \[
% \sC(X) = \bigl\{ Q \in X : \op{length}( \langle P, Q \rangle \cap X ) \geq 3\ \text{for general }P \in X \bigr\}
% \]
% is nonempty. For each $x \in \sC(X)$, the linear system $\lvert V_{\out} \rvert$ separates tangents at $x$.
% \end{theorem}

In this section, we prove Theorem~\ref{thm:D_outva}~$(2)$. More specifically, we prove that the linear system $\lvert V_{\out}\rvert$ separates tangents vectors at each point $x \in X$. This implies that $\lvert V_{\out}\rvert$ is very ample when $X \subseteq \nP^r$ is not a Roth variety since $\lvert V_{\out}\rvert$ separates distinct two points by \cite[Theorem 4.2]{Il}.
Let $\Lambda^{e-2} \subseteq \nP^r$ be a general linear subspace of codimension $r-(e-2) = n+2$ from which the linear projection induces a finite birational morphism $\pi_{\Lambda,X} \colon X \to \nP^{n+1}$. By definition, $\lvert V_{\out} \rvert$ is the linear subsystem of $\lvert D_{\out} \rvert = \lvert -K_X + (d-n-2)H_X \rvert$ spanned by the divisors $D_{\out}(\pi_{\Lambda,X})$, where $D_{\out}(\pi_{\Lambda,X})$ is the double point divisor of $\pi_{\Lambda,X}$. By Theorem~\ref{thm:D_outva}~$(1)$, $D_{\out}(\pi_{\Lambda,X})$ is always singular at $x$ whenever it passes through $x$. Thus to prove Theorem~\ref{thm:D_outva}~$(2)$, it is necessary to manipulate the linear combinations of geometric sections in the vector space $H^0(X,D_{\out})$.
\medskip
\par One of the key ingredients is the identity \eqref{eq: Mumford identity, abstract form} by Mumford (\cite[Technical Appendix 4]{BM}). For the reader's convenience, we recall the notations here. Let $\Lambda'^{e-1} \subseteq \nP^r$ be a linear subspace of codimension $r-(e-1) = n+1$ which does not intersect $X$. The linear projection from $\Lambda'$ induces a dominant morphism $\pi_{\Lambda',X} \colon X \to \nP^n$. Let $\Ram(\pi_{\Lambda',X})$ be the ramification divisor of $\pi_{\Lambda',X}$. Let $\Lambda \subseteq \Lambda'$ be a linear subspace of codimension one such that the projection from $\Lambda$ induces a finite birational morphism $\pi_{\Lambda,X} \colon X \to \overline{X}_{\Lambda} \subseteq \nP^{n+1}$. The morphism $\pi_{\Lambda',X} \colon X \to \nP^n$ factors as
\[
X \xto{\pi_{\Lambda,X}} \overline{X}_\Lambda \xto{ \pi_{R,\overline{X}_{\Lambda}} } \nP^n,
\]
where the point $R \in \nP^{n+1}$ is the projection image of $\Lambda' \setminus \Lambda$. Fix a homogeneous coordinates $[Z_0,\dots,Z_{n+1}]$ of $\nP^{n+1}$ and let $f_\Lambda \in \nC[\nP^{n+1}] = \nC[Z_0,\dots,Z_{n+1}]$ be the defining equation of $\overline{X}_{\Lambda}$. For a point $Q = [q_0,\dots,q_{n+1}] \in \nP^{n+1}$, define the differential operator $\pd_Q$ as follows:
\[
    \pd_Q := Q \cdot \nabla = q_0 \frac{\pd}{\pd Z_0} + q_1 \frac{\pd}{\pd Z_1} + \dots + q_{n+1} \frac{\pd}{\pd Z_{n+1}}.
\]
We remark that $\pd_Q$ is defined only up to multiplications by nonzero constants. %
The projection $\pi_{\Lambda,\nP^r} \colon \nP^r \dashrightarrow \nP^{n+1}$ induces $\pi_{\Lambda,\nP^r}^* \colon \nC[\nP^{n+1}] \hookrightarrow \nC[\nP^r]$, and it maps $\pd_R f_{\Lambda}$ to the polynomial $F_{\Lambda,\Lambda'} := \pi_{\Lambda,\nP^r}^*( \pd_R f_{\Lambda})$. Now, the identity \eqref{eq: Mumford identity, abstract form} says that the divisor $\Ram(\pi_{\Lambda',X}) + D_{\out}(\pi_{\Lambda,X})$ on $X$ equals to the scheme-theoretic intersection $V(F_{\Lambda,\Lambda'}) \cap X$.

\medskip\par
To obtain a linear combination of geometric sections, we take codimension one linear subspaces $\Lambda_1,\Lambda_2 \subseteq \Lambda'$ and consider linear combination of $a_1 F_{\Lambda_1,\Lambda'} - a_2 F_{\Lambda_2,\Lambda'}$ for $a_1, a_2 \in \nC$. The hypersurface  $V( a_1 F_{\Lambda_1,\Lambda'} - a_2 F_{\Lambda_2,\Lambda'} ) \cap X$ in $X$ is the divisor $\Ram(\pi_{\Lambda', X} ) + D$ on $X$, where $D \in \lvert V_{\out}\rvert$. The proof of Theorem~\ref{thm:D_outva}~$(2)$ consists of the following steps:
\begin{enumerate}
    \item fix a point $x$ and a projective tangent line $L \subseteq \nT_x X$;
    \item take $\Lambda_1,\Lambda_2 \subseteq \Lambda'$ such that $x \not\in V( F_{\Lambda_i,\Lambda} )$ for $i=1,2$;
    \item find suitable $a_1,a_2 \in \nC$ such that $V( a_1 F_{\Lambda_1,\Lambda'} - a_2 F_{\Lambda_2,\Lambda'} )$ contains $x$ but is not tangent to $L$ at $x$.
\end{enumerate}

\medskip\par
Even though the purpose of each step appears simple and straightforward, verifying it requires rather involved procedure. The initial step is to reduce the dimension and codimension, which allows us to assume $X^n \subseteq \nP^r$ is a space curve\,(namely, $(n,r)=(1,3)$).

%
% \par The following is a brief summary of \cite[p.41--43]{BM} which provides an explicit description of double point divisors as a part of a hyperplane section. Let $\Lambda'^{e-1} \subseteq \nP^N$ be a linear subspace of codimension $r-(e-1) = n+1$ which does not intersect $X$. The linear projection from $\Lambda'$ induces a dominant morphism $\pi_{\Lambda',X} \colon X \to \nP^n$. Let $\Ram(\pi_{\Lambda',X})$ be the ramification divisor of $\pi_{\Lambda',X}$. Let $\Lambda \subseteq \Lambda'$ be a linear subspace of codimension one such that the projection from $\Lambda$ induces a birational morphism $\pi_{\Lambda,X} \colon X \to \nP^{n+1}$. The morphism $\pi_{\Lambda',X} \big\vert_X \colon X \to \nP^n$ factors through
% \[
% X \xto{\pi_{\Lambda,X}} \bar{X} \xto{ \pi_{R,\overline{X}} } \nP^n,
% \]
% where the point $R \in \nP^{n+1}$ is the projection image of $\Lambda' \setminus \Lambda$. Fix a homogeneous coordinates $[Z_0,\dots,Z_{n+1}]$ of $\nP^{n+1}$ and let $R = [0,0,\dots,0,1]$. Let $f_{\overline{X}} \in \nC[Z_1,\dots,Z_{n+1}]$ be the defining equation of $\overline{X}$. Then, we have the equality of divisors
% \begin{equation}\label{eq: Mumford identity}
% \Ram(\pi_{\Lambda',X}) + D_{\out}(\pi_{\Lambda,X}) = V\Bigl( \pi_{\Lambda,X}^* \Bigl( \frac{\pd f_{\overline{X}} }{ \pd Z_{n+1} } \Bigr) \Bigr).
% \end{equation}

\subsection{Reduction to space curves}\label{subsec:reduction}
%\textcolor{Blue}{(JP: I moved Lemma 4.2 and Corollary 4.3 to Section 3!)}
We fix $x \in X$ and a tangent line $L \subseteq \nT_x X$. We will reduce $\dim X$ to $1$ by taking hyperplane sections containing $L$. If $n=\dim X \geq 3$, then we can choose a general smooth surface section $S \subseteq \nP^r$ with $L \subseteq \nT_x S$ by Lemma \ref{lem: smoothness of hyperplane sections}. Thus we may reduce the problem to the surface case. If $L \not\subseteq S$, then a general curve section $C \subseteq \nP^r$ with $L =\nT_x C$ is still smooth. However, if $L \subseteq S$, then a general curve section $C \subseteq \nP^r$ with $L =\nT_x C$ is the union of $L$ and a smooth projective curve $C'$ with $x \not\in C'$ (see Corollary \ref{cor: Vout separates tangents, reduction outcome}).

\medskip
To justify the reduction procedure, we also need to compare $V_{\out}$ of $X$ with that of its hyperplane sections.
\begin{proposition}\label{prop: Vout v.a. - reduction to hyperplane sections}
Notations $X, \Lambda, \Lambda', R, f_{\Lambda}, F_{\Lambda, \Lambda'}$ as in Subsection~\ref{subsec: 4.1}. Let $H \subseteq \nP^r$ be a hyperplane and $H_X := X \cap H$. We assume further that $\Lambda' \subseteq H \setminus X$ and that $\pi_{\Lambda, H_X} \colon H_X \to \nP^n$ is finite birational. Let $R_H\in \nP^n$ be the projection of $\Lambda'$ under $\pi_{\Lambda,H} \colon H \dashrightarrow \nP^n$. For the polynomial $f_{\Lambda,H} \in \nC[\nP^n]$ defining $\pi_{\Lambda,H_X}(H_X)$, we associate $F_{\Lambda,\Lambda',H} \in \nC[H]$ with $\pd_{R_H} f_{\Lambda,H}$ via the inclusion $\pi_{\Lambda,H}^* \colon \nC[\nP^n] \hookrightarrow \nC[H]$. Then we have
\[
    V(F_{\Lambda,\Lambda',H}) = V(F_{\Lambda,\Lambda'}) \cap H.
\]
\end{proposition}
\begin{proof}
Let $[Z_0,\dots,Z_r]$ be the homogeneous coordinates for $\nP^r$. Assume $H = V(Z_r)$, $\Lambda' = V(Z_1,Z_2,\dots,Z_n,Z_r)$, and $\Lambda = V(Z_0,Z_1,\dots,Z_n, Z_r)$. The target of the projection $\pi_{\Lambda, X}$ is $\nP^{n+1}$ with the homogeneous coordinates $[Z_0,Z_1,\dots,Z_n,Z_r]$, and the morphism $\pi_{\Lambda,H_X} \colon H_X \to \nP^n$ is simply the restriction of $\pi_{\Lambda,X}$ to $H = V(Z_r)$. It follows that $\pi_{\Lambda, H_X}(H_X) = \pi_{\Lambda,X}(X)\big\vert_H$, or equivalently,
\[
    f_{\Lambda,H}(Z_0,Z_1,\dots,Z_n) = f_{\Lambda}(Z_0,Z_1,\dots,Z_n,0).
\]
By descriptions of $\Lambda'$ and $\Lambda$, the coordinates of $R = \pi_{\Lambda, \nP^r}(\Lambda')$ is $[1,0,\dots,0] \in \nP^{n+1}$. Similarly, the coordinates of $R_H$ is $[1,0,\dots,0] \in \nP^n$. This shows that both $\pd_R$ and $\pd_{R_H}$ are $\pd / \pd Z_0$. It follows that $\pd_{R_H} f_{\Lambda,H} (Z_0,\dots,Z_n) = \pd_{R} f_{\Lambda}(Z_0,\dots,Z_n,0)$ holds, and consequently,
\[
    F_{\Lambda,\Lambda',H}(Z_0,Z_1,\dots,Z_{r-1}) = F_{\Lambda,\Lambda'}(Z_0,Z_1,\dots,Z_{r-1},0).
\]
This proves the proposition.
\end{proof}

Proposition~\ref{prop: Vout v.a. - reduction to hyperplane sections} implies the following. If $H$ is a hyperplane containing $L$ and if $V( a_1 F_{\Lambda_1,\Lambda',H} - a_2 F_{\Lambda_2,\Lambda',H} )$ intersects $L$ transversally at $x$, then so does $V( a_1 F_{\Lambda_1,\Lambda'} - a_2 F_{\Lambda_2,\Lambda'})$. This justifies our reduction procedure. We want to emphasize that the smoothness of $H_X$ is not necessary here. In particular, the reduction procedure works until the last step at which $H_X$ is a (possibly reducible) curve. This is one of the main reasons why we are working with the polynomials $F_{\Lambda,\Lambda'}$ instead of the divisor $\Ram(\pi_{\Lambda',X})+ D_{\out}(\pi_{\Lambda,X})$.

\medskip
\par So far we have reduced the dimension of $X$ to one. To have a space curve, we need to reduce to $e=2$. We use outer projections to achieve it.

\begin{proposition}\label{prop: Vout v.a. - reduction to outer projection}
Notations $X, \Lambda, \Lambda', R, f_{\Lambda}, F_{\Lambda, \Lambda'}$ as in Subsection~\ref{subsec: 4.1}. Assume further that $e \geq 3$\,(so that $\dim \Lambda = e-2 \geq 1$). For a point $T \in \Lambda \setminus X$, let $\pi_{T,X} \colon X \to \nP^{r-1}$ be the projection morphism. Let $X_T = \pi_{T,X}(X)$ and let $\Lambda_T,\, \Lambda'_T \subseteq \nP^{r-1}$ be the projection images of $\Lambda,\, \Lambda'$ under $\pi_{T,\nP^r} \colon \nP^r \dashrightarrow \nP^{r-1}$ respectively. For the inclusion morphism $\pi_{T,\nP^r}^* \colon \nC[ \nP^{r-1}] \hookrightarrow \nC[\nP^r]$, the following holds:
\[
    \pi_{T,\nP^r}^* ( F_{\Lambda_T,\Lambda_T'} ) = F_{\Lambda,\Lambda'}.
\]
\end{proposition}
\begin{proof}
Since $\pi_{\Lambda,\nP^r} = \pi_{\Lambda_T,\nP^{r-1}} \circ \pi_{T,\nP^r}$, we have
\[
    F_{\Lambda,\Lambda'} = \pi_{\Lambda,\nP^r}^*( \pd_R f_{\Lambda}) = \pi_{T,\nP^r}^* (\pi_{\Lambda_T,\nP^{r-1}}^* (\pd_R f_{\Lambda})) = \pi_{T,\nP^r}^* F_{\Lambda_T,\Lambda_T'}. \qedhere
\]
\end{proof}

From Proposition~\ref{prop: Vout v.a. - reduction to outer projection}, it is immediate to see that if $T \in \Lambda \setminus (X \cup L)$ and if $V( a_1 F_{(\Lambda_1)_T,\Lambda'_T} - a_2 F_{(\Lambda_2)_T,\Lambda_T'} ) \subseteq \nP^{r-1}$ intersects $L$ trasversally at $x$, then so does $V(a_1 F_{\Lambda_1,\Lambda'} - a_2 F_{\Lambda_2,\Lambda'})$.

\medskip\par
In summary, to prove Theorem~\ref{thm:D_outva}~$(2)$, it suffices to consider the following two cases: 
\begin{enumerate}
    \item $X \subseteq \nP^3$ is smooth curve and $x \in X$, or
    \item $X$ is a union of $\nT_x X$ and a smooth curve $C' \not\ni x$.
\end{enumerate}
As we mentioned in Remark \ref{rem:Ilic'sproof}, in the case $(1)$, it was known that $D_{\out}$ is very ample, but it was not known that $\lvert V_{\out} \rvert$ is very ample.

\subsection{Separation of tangents}

From now on, we assume $X \subseteq \nP^3$ is either a smooth curve or a union of $\nT_x X$ and a smooth curve $C'\not\ni x$. Let $\Lambda' \subseteq \nP^3 \setminus X$ be a linear space of dimension $e-1 = 1$\,(i.e. a line) and let $P=\Lambda_1,\, Q=\Lambda_2 \in \Lambda'$ be distinct points. We need some generosity assumptions on the points $P$ and $Q$. 
%We will describe the conditions explicitly in the definition below.

\begin{definition}\label{def: birational centers and admissible pairs}
Let $x$ be a point of $X$. We call a point $P \in \nP^3$ a \emph{birational center of $(X,x)$} if it fulfills the following two conditions:
\begin{enumerate}
\item $P \not\in \sB(X) = \bigl\{ Q \in \nP^3 \setminus X \,\mid\, \op{length}( \langle P,Q \rangle \cap X ) \geq 2\ \text{for general }P 
\in X \bigr\}$.
\item $P \not\in \sC_x = \op{Cone}(x,X) =  \overline{\bigcup_{y \in X \setminus \{x\}} \langle x, y \rangle}$.
\end{enumerate}
In addition, a pair $P,Q$ of birational centers of $(X,x)$ is said to be \emph{admissible} if it satisfies the following two conditions:
\begin{enumerate}\setcounter{enumi}{2}
\item $\langle P,Q \rangle \cap X = \emptyset$.
\item $\langle P, Q \rangle \cap \mathbf{T}_x X = \emptyset$.
\end{enumerate}
\end{definition}

\begin{proposition}\label{prop: consequences of generosity assumptions}
Let $P$ be a birational center of $(X,x)$. Then, the projection morphism $\pi_{P,X} \colon X \to \nP^2$ is finite and birational. Moreover, $x \not\in D_{\out}(\pi_{P,X})$. If $P,Q$ is an admissible pair of birational centers of $(X,x)$, then the morphism $\pi_{\langle P,Q \rangle, X} \colon X \to \nP^1$ is finite and $x \not\in \Ram(\pi_{\langle P,Q \rangle,X})$.
\end{proposition}

\begin{proof}
The proposition immediately follows from the definition. 
\end{proof}

It is important to point out that an admissible pair always exists in our situation.
\begin{proposition}\label{prop: admissible pairs, existence}
An admissible pair of birational centers of $(X,x)$ exists.
\end{proposition}
\begin{proof}  
We need to consider two cases, that is, either $X$ is smooth or $X = \nT_xX \cup C'$ with $C'$ smooth and $x \not\in C'$. Suppose $X$ is smooth. The closure $\bar{\sB}(X)$ of $\sB (X)$ is a union of linear spaces of dimension at most $n-1 = 0$\,(see \cite[Theorem~1]{N2}), and $\dim \sC_x = 2$. Hence, $P \not\in \sB(x)$ and $P \not \in \sC_x$ for general $P$. This shows that the general point in $\nP^3$ is a birational center of $(X,x)$. Moreover, both $X$ and $\nT_xX$ are of codimension $2$ in $\nP^3$, a general line does not intersect neither $X$ nor $\nT_xX$. This shows that for general $P,Q$, we have $\langle P,Q \rangle \cap X = \langle P, Q \rangle \cap \nT_x X = \emptyset$.

\medskip\par
Now, assume $X = \nT_x X \cup C'$, and $Q \in \sB(X)$. For general $P \in C'$, the line $\langle P, Q\rangle $ does not intersect $\nT_xX$; otherwise, $C' \subseteq \langle Q, \nT_x X\rangle$, i.e. $X$ is degenerate. Hence, $\length(\langle P,Q \rangle \cap C' ) \geq 2$ for general $P \in C'$. This shows $\sB(X) \subseteq \sB(C')$. We invoke \cite[Theorem~1]{N2} to see that $\sB(X)$ is zero-dimensional. In particular, the condition (1) holds for general $P \in \nP^3$. To verfiy (2) for general $P$, we observe that $\sC_x = \op{Cone}(x,X) = \op{Cone}(x, \nT_xX) \cup \op{Cone}(x, C') = \nT_xX \cup \op{Cone}(x,C')$. Since $\dim sC_x = 2$, a general $P$ is not contained in $\sC_x$. Conditions (3) and (4) follows immediately from the dimension counting.
\end{proof}

From now on, we assume that $P,Q \in \Lambda'$ are admissible pair of birational centers of $(X,x)$. Using notations in Subsection~\ref{subsec: 4.1}, we have
\[
    \Ram(\pi_{\Lambda',X}) + D_{\out}(\pi_{P,X}) = V( F_{P,\Lambda'}) \cap X.
\]
Let $a_{P,\Lambda'}, \, a_{Q,\Lambda'} \in \nC$ be complex numbers such that
\[
    \frac{a_{P,\Lambda'}}{a_{Q,\Lambda'}} = \frac{ F_{P,\Lambda'}}{F_{Q,\Lambda'}}(x),
\]
regarding $F_{P,\Lambda'} / F_{Q,\Lambda'}$ as a rational function on $\nP^3$. If we put
\[
    F_{P,Q} := a_{P,\Lambda'} F_{Q,\Lambda'} - a_{Q,\Lambda'} F_{P,\Lambda'},
\]
then $V(F_{P,Q}) \cap X = \Ram(\pi_{\Lambda',X}) + D$, where $D \in \lvert V_{\out} \rvert$.

\medskip\par
We need to investigate when the divisor $D$ is tangent to $\mathbf{T}_x X$ at $x$. Parametrize $\mathbf{T}_x X$ by $\{x + t v\}_{t \in \nP^1}$, where $v \in \mathbf{T}_x X \setminus \{x \}$. Then,
\begin{align}
\begin{split}
\text{$D$ is tangent to $\mathbf{T}_x X$ at $x$}\ \ & \Longleftrightarrow  \  \ (v \cdot \nabla) F_{P,Q}(x) = 0  \label{eq: separated tangents 1} \\
&\Longleftrightarrow\  \ v \cdot \bigl( a_{P,\Lambda'} \nabla F_{Q,\Lambda'}- a_{Q,\Lambda'} \nabla F_{P,\Lambda'} \bigr)(x) = 0.
\end{split}
\end{align}
In above, the products are the standard inner product of vectors. The last condition in (\ref{eq: separated tangents 1}) can be reinterpreted as follows. Let $H_{P,\Lambda'} \subseteq \nP^3$ be the plane defined by the equation
\[
\sum_{i=0}^3 Z_i \frac{\pd F_{P,\Lambda'}}{\pd Z_i}(x) = 0.
\]
The plane $H_{P,\Lambda'}$ intersects $\mathbf{T}_x X$ at $x + t v$\,($t \in \nP^1$) if and only if
\[
\sum_{i=0}^3 (x_i+t v_i) \frac{\pd F_{P,\Lambda'}}{\pd Z_i}(x)  = (\deg F_{P,\Lambda'}) F_{P,\Lambda'}(x) + t (v \cdot \nabla) F_{P,\Lambda'}(x) = 0.
\]
Note that $F_{P,\Lambda'}(x)\neq 0$, so the above equation must have a unique solution $t = t_0$. Since $\deg F_{P,\Lambda'} = \deg X - 1 = \deg F_{Q,\Lambda'}$, the last condition in (\ref{eq: separated tangents 1}) is equivalent to $ H_{P,\Lambda'} \cap \nT_x X \neq H_{Q,\Lambda'} \cap \nT_x X$. Consequently, we have
\begin{equation}
\text{$D$ is not tangent to $\nT_x X$ at $x$}\ \ \Longleftrightarrow\ \ H_{P,\Lambda'} \cap \nT_x X \neq H_{Q,\Lambda'} \cap \nT_x X. \label{eq: separated tangents 2}
\end{equation}
\subsection{Proof of the right hand side of (\ref{eq: separated tangents 2})} Fix coordinates in $\nP^3$ such that $P = [1,0,0,0] \in \nP^3$ and that $\nT_xX$ contains a point $T = [1,1,0,0] \neq x$. If $Q = [1,q_1,q_2,q_3]$, then
\[
F_{P,\Lambda'} = \pi_{P,\nP^3}^*\bigl( \pd_{Q_P}f_P \bigr) = q_1 \frac{\pd f_P}{\pd Z_1} + q_2 \frac{\pd f_P}{\pd Z_2} + q_3 \frac{\pd f_P}{\pd Z_3}.
\]
Here, the pullback $\pi_{P,\nP^3}^*$ is suppressed in the notation since $\pi_{P,\nP^3}^*$ is nothing but the natural inclusion $\nC[Z_1,Z_2,Z_3] \hookrightarrow \nC[Z_0,Z_1,Z_2,Z_3]$.

\medskip\par
The plane $H_{P,L}$ passes through $T$ if and only if
\begin{equation}\label{eq: H thru T, general form}
( T \cdot \nabla) F_{P,\Lambda'}(x) = \frac{\pd F_{P,\Lambda'}}{\pd Z_1}(x) = 0,
\end{equation}
or equivalently,
%:
\begin{equation}\label{eq: H thru T}
q_1 \frac{\pd^2 f_P}{\pd Z_1^2}(x) + q_2 \frac{\pd^2 f_P}{\pd Z_1 \pd Z_2}(x) + q_3 \frac{\pd^2 f_P}{\pd Z_1 \pd Z_3}(x).
\end{equation}
%:
The condition (\ref{eq: H thru T}) is a single linear condition on $Q = [1,q_1,q_2,q_3]$ unless $\dfrac{\pd^2 f_P}{\pd Z_1 \pd Z_i}(x)=0$ for all $i=1,2,3$.
%:
\medskip\par
From now on, we denote by $\mathcal L_{P,T} \subseteq \nP^3$ the set of points $[q_0,q_1,q_2,q_3]$ defined by the equation \eqref{eq: H thru T, general form}. In general, $\mathcal L_{P,T}$ is a plane and parametrizes the points $Q$ satisfying $T \in H_{P, \langle P, Q \rangle}$. In Proposition~\ref{prop: inflection points and L_PT} and Proposition~\ref{prop: inflection and osculation}, we will provide a geometric illustration of the condition $\mathcal L_{P,T} = \nP^3$. %
Also, \eqref{eq: H thru T} suggests to define $\mathcal L_{P,x} = \langle P, \mathbf{T}_x X \rangle$. 

\medskip\par
Assume there exist admissible birational centers $P, Q$ such that 
\begin{equation}\label{eq: goal}
Q \in \mathcal L_{P,T},\ \text{but\ } P \not\in \mathcal L_{Q,T}.
\end{equation}
Then, the right hand side of \eqref{eq: separated tangents 2} follows immediately. From now on, we assume
\begin{equation}\label{eq: contrary assumption}
% \begin{split}
% Q \in \mathcal L_{P,T}\ \Longleftrightarrow \ & P \in \mathcal L_{Q,T} \text{ whenever } P, Q \text{ are}\\
% &\text{admissible birational centers}.
% \end{split}
\text{if $P,Q$ are admissible birational centers, then}\ \ Q \in \mathcal L_{P,T} \Longleftrightarrow P \in \mathcal L_{Q,T},
\end{equation}
and finally we will derive a contradiction and hence verify the right hand side of \eqref{eq: separated tangents 2}.

\medskip\par
Before getting into the technical proof, we give some basic observation about $\mathcal L_{P,T}$.

\begin{lemma}\label{lem: when L_PT is not a pencil}
The condition (\ref{eq: H thru T}) is nontrivial for general $T \in \mathbf{T}_x X \setminus\{x\}$. Moreover, the condition is trivial at most one point of $\mathbf{T}_x X \setminus \{x\}$, and in such a case all the other $\mathcal L_{P,T}$ are $\langle P, \mathbf{T}_x X \rangle $.
\end{lemma}
%:
\begin{proof}
Assume that (\ref{eq: H thru T}) is trivial at the point $T$, and pick another $T' = T + tx$ where $t \neq 0$. Then, (\ref{eq: H thru T}) with respect to $T'$ is written as
\[
\pd_{T'}F_{P,\Lambda'}(x) = (\pd_T + t \pd_x)F_{P,\Lambda'}(x) = 0
\]
Here, $\pd_T F_{P,\Lambda'}(x) = \dfrac{\pd F_{P,\Lambda'}}{\pd Z_1}(x)$ is the left hand side of (\ref{eq: H thru T}), and
\[
\pd_x F_{P,\Lambda'}(x) = (\deg F_{P,\Lambda'}) F_{P,\Lambda'}(x) = (\deg F_{P,\Lambda'}) \biggl( q_1 \frac{\pd f_P}{\pd Z_1}(x) + q_2 \frac{\pd f_P}{\pd Z_2}(x) + q_3 \frac{\pd f_P}{\pd Z_3}(x) \biggr).
\]
This is nontrivial since the curve $\overline{X}_P = \pi_{P,X}(X)$ is nonsingular at $\bar x_P = \pi_{P,X}(x)$. It follows that the condition $\pd_{T'}F_{P,\Lambda'}(x)=0$ differs from $\pd_{T} F_{P,\Lambda'}(x)=0$, hence cannot be trivial. The last statement follows from the fact that $\pd_x F_{P,\Lambda'}(x) = 0$ defines $\pi_{P,\nP^3}^* (\mathbf{T}_{x_P}X_P) = \langle P, \mathbf{T}_x X \rangle$.
\end{proof}

\begin{lemma} $ x \in \mathcal L_{P,T}$.\end{lemma}
\begin{proof}
Let $x = [x_0,x_1,x_2,x_3]$. It suffices to prove that $(q_1,q_2,q_3) = (x_1,x_2,x_3)$ is a solution to (\ref{eq: H thru T}). Indeed,
\begin{align*}
x_1 \frac{\pd^2 f_P}{\pd Z_1^2}(x) + x_2 \frac{\pd^2 f_P}{\pd Z_1 \pd Z_2}(x) + x_3 \frac{\pd^2 f_P}{\pd Z_1 \pd Z_3}(x) &{}= ((x_1,x_2,x_3) \cdot \nabla) \frac{\pd f_P}{\pd Z_1}(x) \\
&{}= (\deg f_P - 1) \cdot \frac{\pd f_P}{\pd Z_1}(x),
\end{align*}
and the latter term is zero since $\dfrac{\pd}{\pd Z_1}$ is a tangent vector of $\overline{X}$ at $\bar x$.
\end{proof}
%:

The following proposition interprets Lemma~\ref{lem: when L_PT is not a pencil} geometrically.

\begin{proposition}\label{prop: inflection points and L_PT}
Let $P \in \nP^3$ be a birational center, and let $\overline{X}_P = \pi_{P,X}(X)$ and $\bar x_P = \pi_{P,X}(x)$. Then, the following are equivalent:
\begin{enumerate}
    \item $\bar x_P$ is an inflection point of $\overline{X}_P$;
    \item there exists a point $T_P \in \mathbf{T}_x X$ such that $\mathcal L_{P, T_P} = \nP^3$.
\end{enumerate}
Moreover, if (2) holds, then such $T_P$ is unique.
\end{proposition}

\begin{proof}
Let $f_P \in \nC[Z_0,Z_1,Z_2]$ be the equation of $\overline{X}_P \subseteq \nP^2$, and let $\mathbf H := \op{Hess} (f_P)_{\bar x_P}$ be the Hessian matrix of $\bar f_P$ evaluated at $\bar x_P$. We assume that $\bar x_P \in \overline{X}_P$ is an inflection point, and show that $\op{rank} \mathbf H = 2$. Taking suitable coordinates, we assume $\bar x_P = [1,0,0]$ and $\mathbf{T}_{\bar x_P} \overline{X}_P = (Z_2 = 0)$. Let $\gamma \colon \Delta \to \overline{X}_P$, $t \mapsto [ \gamma_0(t), \gamma_1(t), \gamma_2(t)]$ be a local parametrization of $\overline{X}_P$ at $\gamma(0) = \bar x_P$. Then,
\[
  \displaystyle \frac{ \pd f_P} {\pd Z_i}(\bar x_P) = 0 \Longleftrightarrow i \neq 2
\]
since $\displaystyle \mathbf{T}_{\bar x_P} \overline{X}_P = \Bigl( \sum_{i=0,1,2} Z_i \frac{ \pd f_P } { \pd Z_i }( \bar x_P ) = 0 \Bigr) = (Z_2 = 0 )$. By Euler's theorem,
\[
\sum_{j=0,1,2} \gamma_j(0) \frac{\pd^2 f_P }{\pd Z_j \pd Z_i} (\bar x_P ) = (\deg f_P -1 ) \frac{\pd f_P}{\pd Z_i} ( \bar x_P ),\quad \forall i \in \{0,1,2\},
\]
and the left hand side of the above equation reads $\displaystyle\gamma_0(0) \frac{\pd^2 f_P}{ \pd Z_j \pd Z_i} (\bar x_P)$. It follows that
\[
\frac{\pd^2 f_P}{ \pd Z_j \pd Z_i} (\bar x_P) = \left\{
\begin{array}{ll}
 0 & i = 0,1 \\
 \neq 0 & i = 2.
\end{array}
\right.
\]
Since the matrix $\mathbf H$ is singular, $\mathbf {H}$ is of the form
\[
\left( 
\begin{array}{ccc}
0 & 0 & \neq 0 \\
0 & 0 & ? \\
\neq 0 & ? & ? 
\end{array}
\right),
\]
showing that $\op{rank} \mathbf H = 2$. Let $\bar T \in \mathbf{T}_{\bar x_P} \overline{X}_P \setminus \{ \bar x_P \}$ be a point, and let $T \in \mathbf{T}_x X$ be the unique point which maps to $\bar T$ along $\pi_{P,X}$. Then, $\mathcal L_{P,T}$ is the pullback of the line defined by the matrix equation
\[
    \{ [Z_0,Z_1,Z_2] \in \nP^3 \mid (\ Z_0\ Z_1\ Z_2\ )\, \mathbf H\, (\ \bar T_0\ \bar T_1\ \bar T_2\ )^{\sf T} = 0 \},\quad\text{(cf. Eq.~\eqref{eq: H thru T})}
\]
where $\bar T = [ \bar T_0, \bar T_1, \bar T_2 ]$ and ${({}\mathrel{\cdot}{})}^{\sf T}$ denotes the matrix transpose. In particular, we have
\begin{align*}
\mathcal L_{P,T} = \nP^3\  & \Longleftrightarrow\ \mathbf H (\ \bar T_0\ \bar T_1\ \bar T_2\ )^{\sf T} = 0 \\
& \Longleftrightarrow\ (\ \bar T_0\ \bar T_1\ \bar T_2\ )^{\sf T} \in \ker \mathbf {H},
\end{align*}
and there is a unique $T = T_P$ satisfying the above equivalent conditions since $\dim \ker \mathbf H = 1$. This shows $(1) \Rightarrow (2)$.
Conversely, if (2) is true, then $T_P$ induces the nontrivial vector of $\ker \mathbf {H}$, hence (1) follows. Such $T_P$ is unique since $\dim \ker \mathbf {H} = 1$.
\end{proof}

\begin{proposition}\label{prop: inflection and osculation}
Notations as in Proposition~\ref{prop: inflection points and L_PT}. The curve $\overline{X}_P$ has an inflection point at $\bar x_P$ if and only if $P$ lies on a plane $\Pi$ osculating to $X$ at $x$. In particular, if $x \in X$ is an inflection point, then so is $\bar x_P \in \overline{X}_P$ for any birational center $P$.
\end{proposition}
\begin{proof}
Let $\gamma \colon \Delta \to X$ be a local parametrization at $x = \gamma(0)$. By choosing homogeneous coordinates $[Z_0,Z_1,Z_2,Z_3]$ in such a way that $x = [1,0,0,0]$, $\mathbf{T}_x X = (Z_2 = Z_3 = 0)$, and $P = [0,0,0,1]$, we may assume $\gamma(0) = [1,0,0,0]$ and $\gamma'(0) = [0,1,0,0]$. After the projection, the induced map $\bar \gamma_P := \pi_{P,\nP^3} \circ \gamma = [ \gamma_0, \gamma_1, \gamma_2]$ parametrizes $\overline{X}_P$ at $\bar x_P$. We have the following equivalences:
\begin{align*}
\bar x_P \in \overline{X}_P\text{ is an inflection point\ }&\Longleftrightarrow\ 
\bar\gamma_P''(0) \in [ \bar\gamma_P(0) \wedge \bar\gamma_P'(0) ] \\
&\Longleftrightarrow\ \gamma_2''(0) = 0.
\end{align*}
On the other hand, we have
\begin{align*}
\Pi = \langle P, \mathbf T_x X \rangle\text{\ osculates to\ $X$ at $x$\ }&\Longleftrightarrow\ 
\nP[ \gamma(0) \wedge \gamma'(0) \wedge \gamma''(0) ] \subseteq \Pi \\
&\Longleftrightarrow\ \gamma_2''(0)=0,
\end{align*}
and hence, the desired equivalence follows. If $x \in X$ is an inflection point, then every plane containing $\mathbf{T}_x X$ osculates to $X$ at $x$, and thus, $\bar x_P \in \overline{X}_P$ is an inflection point for any birational center $P$.
\end{proof}
%:

\begin{proposition}\label{prop: union of L_PT cap L_QT} Assume that $x \in X$ is not an inflection point. Let $R, S \in \nP^3$ be an admissible pair of birational centers such that neither $\langle R, \mathbf T_x X\rangle$ nor $\langle S , \mathbf T_x X \rangle$ osculates to $X$ at $x$. Since $\{ \mathcal L_{R,T} \mid T \in \mathbf T_x X \}$ is a pencil of planes\,(cf. Lemma~\ref{lem: when L_PT is not a pencil} and Proposition~\ref{prop: inflection points and L_PT}), we may take $T_0 \in \mathbf{T}_x X \setminus\{ x\}$ such that $S \in \mathcal L_{R,T_0}$. Fix a vector $\hat R$\,(resp. $\hat S$, $\hat x$, etc.) in $\nC^4$ representing $R \in \nP^3$\,(resp. $S$, $x$, etc.). 
Then, using Pl\"ucker embedding, we may write
\[
\mathcal L_{R,T} = \nP\bigl[ \, \hat x \wedge \hat R \wedge ( \hat T_0 + t \gamma \hat S ) \, \bigr]\quad \text{and\quad }%
\mathcal L_{S,T} = \nP\bigl[ \, \hat x \wedge \hat S \wedge ( \hat T_0 + t \gamma' \hat R) \, \bigr],
\]
where $\gamma,\gamma'$ are nonzero constants and $t$ is an element of $\nP^1 = \nC \cup\{ \infty\}$ such that $\hat T = \hat x + t \hat T_0$. Moreover,
\[
\bigcup_{T \in \mathbf{T}_x X} (\mathcal L_{R,T} \cap \mathcal L_{S,T}) = \langle x, R, S \rangle \cup \Theta,
\]
where $\Theta$ is the plane $\nP [ \hat x \wedge ( \gamma' \hat R + \gamma \hat S ) \wedge \hat T_0 ]$.
\end{proposition}
\begin{proof}
Since $x \in X$ is not an inflection point, the assumption \eqref{eq: contrary assumption} implies that $\mathcal L_{R,T} = \mathcal L_{S,T} = \langle x, R, S \rangle$. Since $\mathcal L_{R,x} = \langle R, \mathbf{T}_x X \rangle$, we may write
\begin{align*}
\mathcal L_{R,x} &= \nP [ \hat x \wedge \hat R \wedge \hat \ell_0 ], \\
\mathcal L_{R,T_0} &= \nP [ \hat x \wedge \hat R \wedge \hat \ell_\infty],
\end{align*}
for $\ell_0 = \alpha_0 \hat x + \beta_0 \hat R + \gamma_0 \hat T_0$ and $\ell_\infty = \alpha_\infty \hat x + \beta_\infty \hat R + \gamma_\infty \hat S$ with $\gamma_0,\gamma_\infty \neq 0$. Let $\hat T = \hat x + t \hat T_0$ for $t \in \nC \cup \{0\}$. We may assume that each member of the pencil $\{\mathcal L_{R,T}\}_T$ is described by
\begin{align*}
 \mathcal L_{R,T} &= \nP[ \hat x \wedge \hat R \wedge ( \hat \ell_0 + t \hat \ell_\infty ) ] \\
 &= \nP [ \hat x \wedge \hat R \wedge ( \gamma_0 \hat T_0 + t \gamma_\infty \hat S ) ].
\end{align*}
Similarly, we have
\[
    \mathcal L_{S,T} = \nP [ \hat x \wedge \hat S \wedge ( \gamma_0' \hat T_0 + t \gamma_\infty' \hat R ) ]
\]
for some nonzero constants $\gamma_0', \gamma_\infty'$. Wihtout loss of generality, we may assume $\gamma_0 = \gamma'_0 = 1$. Then,
\[
    \mathcal L_{R,T} \cap \mathcal L_{S,T} = \left\{ 
    \begin{array}{ll}
    \nP [ \hat x \wedge ( \hat T_0 + t \gamma_\infty \hat S + t \gamma_\infty' \hat R)  ] & t \neq \infty \\
    \nP [ \hat x \wedge \hat R \wedge \hat S ] & t = \infty.
    \end{array}
    \right.{}
\]
Taking union over all $t \in \nP^1$ yields
\[
    \bigcup_{T \in \mathbf{T}_x X} ( \mathcal L_{R,T} \cap \mathcal L_{S,T} ) = \langle x, R, S \rangle \cup \nP [ \hat x \wedge ( \gamma'_\infty \hat R + \gamma_\infty \hat S ) \wedge \hat T_0 ]. \qedhere
\]
\end{proof}

\begin{proposition}\label{prop: V_out separates tangents, non-osculating case}
Notations as in Proposition~\ref{prop: union of L_PT cap L_QT}. Let $\Theta$ be the plane component of 
\[
    \bigcup_{T \in \mathbf{T}_x X} (\mathcal L_{R,T} \cap \mathcal L_{S,T} )
\]
which differs from $\langle x, R, S \rangle$. If $\Theta$ is not osculating to $X$ at $x$, then the assumption \eqref{eq: contrary assumption} fails to hold. In particular, the right hand side of \eqref{eq: separated tangents 2} holds for some admissible pair $P,Q$.
\end{proposition}
\begin{proof}
Let $\Pi$ be the plane osculating to $X$ at $x$, and let $P$ be a general point of $\Theta \setminus ( \langle x, R, S \rangle \cup \Pi )$. Then, $P$ belongs to $\mathcal L_{R,T_0} \cap \mathcal L_{S,T_0}$ for some $T_0 \in \mathbf{T}_x X$. We claim that both of the pairs $\{P,R\}$ and $\{P,S\}$ are admissible birational centers.

\medskip\par
We need to consider to cases, that is, either $X$ is smooth or $X = \nT_x X \cup C'$ with $C'$ smooth and $x \not\in C'$. Suppose $X$ is smooth. We examine the locus of $P$ in which the pair $P,R$ is admissible. %
It is obvious that the conditions (1), (2), and (3) of Definition~\ref{def: birational centers and admissible pairs} cannot exclude a plane. The condition (4) might exclude a plane $\langle R, \nT_x X \rangle$. However, by Proposition~\ref{prop: union of L_PT cap L_QT}, $\Theta$ contains neither $R$ nor $S$. These arguments show that the admissibility condition cannot exclude $\Theta$. Thus, the pair $P,R$ are admissible for some $P$\,(hence for general $P$ as well).

\medskip\par
Consider the case $X = \nT_x X \cup C'$. As we have seen in the proof of Proposition~\ref{prop: admissible pairs, existence}, $\sB(X) \subseteq \sB(C')$ and $\sC_x = \nT_x X \cup \op{Cone}(x,C')$. Hence, neither $\sB(X)$ nor $\sC_x$ contains a plane. The conditions (3) and (4) of Definition~\ref{def: birational centers and admissible pairs} may exclude the plane $\langle R, \nT_xX \rangle$, but as in the smooth case this plane differs from $\Theta$. This proves that $P,R$ is admissible for general $P$.

\medskip \par
In both cases, we saw that $P,R$ is an admissible pair for general $P \in \Theta$. Same arguments apply for $P,S$. By Proposition~\ref{prop: inflection points and L_PT} and Proposition~\ref{prop: inflection and osculation}, we have $\mathcal L_{P,T_0} \neq \nP^3$. From the assumption \eqref{eq: contrary assumption} we derive a contradiction that $\mathcal L_{P,T_0} \supseteq \langle x, P, R, S \rangle = \nP^3$.
\end{proof}

When $X = \nT_x X \cup C'$ as in the case (2) in the end of Subsection \ref{subsec:reduction}, there is no osculating plane to $X$ at $x$. Thus Proposition \ref{prop: V_out separates tangents, non-osculating case} finishes the proof for this case. We henceforth assume that $X$ is smooth as in the case (1) in the end of Subsection \ref{subsec:reduction}. It still remains to consider the following two cases: 
\begin{enumerate}\renewcommand*{\labelenumi}{(\roman{enumi})}
\item either the plane $\Theta$ osculates to $X$ at $x$ for every $R,S$ as in Proposition~\ref{prop: union of L_PT cap L_QT}, or
\item $x \in X$ is an inflection point.
\end{enumerate}
The cases $(\mathrm{i})$ and $(\mathrm{ii})$  will be completed in Proposition \ref{prop: Vout separates tangnets, non-inflectionary case} and Corollary \ref{cor: Vout separates tangnets, inflectionary case}. We remark that a point $x \in \sC(X)$ may be an inflection point (see Example \ref{ex:curvecase} $(1)$).
These cases are technically involved, so we need several preparatory steps.
% \cmt{Revise from here} We shall prove the following.
% \begin{lemma}\label{lem: Key lemma (V_out separates tangents), osculating case}
% Let $P$ be a birational center, and let $L = \langle P, x \rangle$. Assume either of the following.
% \begin{enumerate}
%     \item $x \in X$ is an inflection point;
%     \item $x \in X$ is not an inflection point, but the plane $\Xi := \langle L, \mathbf T_x X \rangle$ osculates to $X$ at $x$.
% \end{enumerate}
% For each birational center $P \in L$, let $T_P \in T_x X$ be the point such that $\mathcal L_{P,T_P} = \nP^3$\,(see Proposition~\ref{prop: inflection points and L_PT}). Then, the map $P \mapsto T_P$ is not constant.
% \end{lemma}
% This lemma is technically involved and requires several preparatory steps.

\begin{proposition}\label{prop: coefficient comparison for projection from P, general case}
Let $[Z_0,Z_1,Z_2,Z_3]$ be homogeneous coordinates for $\nP^3$, and let
\[
\gamma := [ \gamma_0, \gamma_1, \gamma_2, \gamma_3] \colon \Delta \longrightarrow \nP^3
\]
be a local parametrization of $X$ at $x = \gamma(0)$. Without loss of generality, we put $x = \gamma(0) = [1,0,0,0]$ and $\gamma'(0) = [0,1,0,0]$, so that $\mathbf T_x X = (Z_2 = Z_3 = 0)$. Let $L = (Z_1 = Z_2 = 0)$ be a line passing through the point $x$. Let $P = [p_0, 0, 0, p_3] \in L$ be a birational center. The projection map $\pi_{P,\nP^3} \colon \nP^3 \dashrightarrow \nP^2$ is given by
\[
    [Z_0,Z_1,Z_2,Z_3] \longmapsto [W_0, W_1, W_2 ] := [p_3 Z_0 - p_0 Z_3, Z_1, Z_2 ].
\]
Let $\overline{X}_P$\,(resp. $\bar x_P$) be the projection image of $X$\,(resp. $x$) under $\pi_{P,\nP^3}$, and let $f_P \in \nC [W_0,W_1,W_2]$ be the defining equation of $\overline{X}_P$. 
% The map 
% \[
% \bar\gamma_P(t) = [ \bar \gamma_0(t), \bar \gamma_1(t) , \bar \gamma_2(t) ] := [ p_3 \gamma_0(t) - p_0 \gamma_3(t) , \gamma_1(t), \gamma_2(t) ] 
% \] 
% gives a local parametrization of $\overline{X}_P$ at $\bar x_P = \bar\gamma_P(0)$. Since $\bar x_P \in \overline{X}_P$ is an inflection point, $\lambda := \length( \overline{X}_P \cap \mathbf{T}_{\bar x_P} \overline{X}_P ) \geq 3$. The tangent line $\mathbf{T}_{\bar x_P} \overline{X}_P$ is given by $(W_2=0)$, hence we may write
If $\lambda := \length( \overline{X}_P \cap \mathbf{T}_{\bar x_P} \overline{X}_P )$, then we may write
\[
    f_P(W_0,W_1,W_2) = W_1^\lambda \varphi(W_0,W_1) + W_2 \psi ( W_0, W_1, W_2),
\]
where $\varphi(W_0,W_1) = \sum_{i,j} \varphi_{ij} W_0^i W_1^j$ and $\psi(W_0,W_1,W_2) = \sum_{i,j,k} \psi_{ijk} W_0^i W_1^j W_2^k$ for $\varphi_{ij}, \psi_{ijk} \in \nC$. Then, we have the following identities
\begin{equation}\label{eq: coefficient comparison of f(r(t))=0}
\left\{
\begin{array}{l}
    \displaystyle \varphi_{d-\lambda,0} + \psi_{d-1,0,0} \frac{ \gamma_2^{(\lambda)}}{\lambda!}  = 0, \\
    \displaystyle \lambda \frac{\gamma_1''}{2!}\varphi_{d-\lambda,0} + \varphi_{d-\lambda-1,1} + \frac{ \gamma_2^{(\lambda)}}{\lambda!} \psi_{d-2,1,0} +  \frac{ \gamma_2^{(\lambda+1)}}{(\lambda+1)!} \psi_{d-1,0,0}  = 0.
\end{array}\right.
\end{equation}
Here, $d = \deg f_P$ and the superscript ${(\,\cdot\,)}^{(k)}$ denotes the $k$th order derivative evaluated at $t=0$.
\end{proposition}
\begin{proof}
The local parametrization $\bar \gamma_P := \pi_{P,\nP^3} \circ \gamma$ of $\overline{X}_P$ at $\bar x_P$ is given by
\[
    \bar \gamma_P (t) := [ \bar \gamma_0(t), \bar\gamma_1(t), \bar\gamma_2(t)],
\]
where
\begin{equation}
\begin{split}
\bar \gamma_0(t) &= p_3 \gamma_0(t) - p_0 \gamma_3(t) = 1 + \frac{1}{2!} \bar \gamma_0'' t^2 + \cdots \\
\bar \gamma_1(t) &= \gamma_1(t) = t + \frac{1}{2!} \gamma_1'' t^2 + \cdots \\
\bar \gamma_2(t) &= \gamma_2(t) = \frac{1}{\lambda !} \gamma_2^{(\lambda)} t^{\lambda} + \frac{1}{( \lambda+1)!} \gamma_2^{(\lambda+1)} + \cdots.
\end{split}
\end{equation}
 Substituting $W_i$ with $\bar\gamma_i$ in $f_P$ yields
\begin{equation*}
\begin{split}
f_P( \bar \gamma_P(t)) & = \left( t + \frac 1 {2!} \gamma_1'' t^2 + \dots + \right)^\lambda \Biggl(  \varphi_{d-\lambda,0} \left( 1 + \frac{ \bar \gamma_0''}{2!} t^2 + \cdots \right)^{d-\lambda} 
\\ & \hspace{4em} + \varphi_{d-\lambda-1, 1} \left( 1 + \frac{\bar \gamma_0''}{2!} t^2 + \cdots \right)^{d-\lambda-1} \left( t + \frac{\gamma_1''}{2!} t^2 + \cdots \right) \cdots \Biggr) 
\\ & \hphantom{{}={}} + \left( \frac{1}{\lambda !} \gamma_2^{(\lambda)} t^\lambda + \frac{1}{(\lambda+1)!} \gamma_2^{(\lambda+1)} t^{\lambda+1} +  \cdots \right) \Biggl( \psi_{d-1,0,0} \left( 1+ \frac{ \bar \gamma_0''}{2!} t^2 + \cdots \right)^{d-1}
\\ & \hspace{4em} + \psi_{d-2,1,0} \left( 1 + \frac{\bar \gamma_0''}{2!} t^2 + \cdots \right)^{d-2} \left( t + \frac{\gamma_1''}{2!}t^2 + \cdots \right) + \cdots \Biggr)
\\ &= \biggl( \varphi_{d-\lambda,0} + \psi_{d-1,0,0} \frac{ \gamma_2^{(\lambda)}}{\lambda!} \biggr) t^\lambda 
\\ & \hphantom{{}={}} + \biggl( \lambda \frac{\gamma_1''}{2!}\varphi_{d-\lambda,0} + \varphi_{d-\lambda-1,1} + \frac{ \gamma_2^{(\lambda)}}{\lambda!} \psi_{d-2,1,0} +  \frac{ \gamma_2^{(\lambda+1)}}{(\lambda+1)!}\psi_{d-1,0,0} \biggr) t^{\lambda+1} + \cdots,
\end{split}
\end{equation*}
and this is identically zero since $\bar\gamma_P$ parametrizes $\overline{X}_P$. By asserting that the coefficients of $t^\lambda$ and $t^{\lambda+1}$ vanish, we obtain \eqref{eq: coefficient comparison of f(r(t))=0}.
\end{proof}

\begin{corollary}\label{cor: kernel of Hessian, osculating or inflectionary case}
Notations as in Proposition~\ref{prop: coefficient comparison for projection from P, general case}. Assume that $\bar x_P \in \overline{X}_P$ is an inflection point. Then, the point $T_P$ satisfying $\mathcal L_{P,T_P} = \nP^3$ is given by
\[
    T_P = \left[-\varphi_{d-\lambda-1,1} -  \frac{\lambda \gamma_1''}{2!} + \frac{ \gamma_2^{(\lambda+1)}}{(\lambda+1) \gamma_2^{(\lambda)} } ,\ d-1,\ 0,\ 0 \right].
\]
\end{corollary}
\begin{proof}
Let $\mathbf H$ be the Hessian matrix of $f_P$ evaluated at $\bar x_P$. The kernel of $\mathbf H$ determines the point $T_P$ at which $\mathcal L_{P, T_P} = \nP^3$\,(see Proposition~\ref{prop: inflection points and L_PT} and its proof). The matrix $\mathbf {H}$ has the form
\[
\left( 
    \begin{array}{ccc}
        0 & 0 & (d-1) \psi_{d-1,0,0} \\
        0 & 0 & \psi_{d-2,1,0} \\
        (d-1)\psi_{d-1,0,0} & \psi_{d-2,1,0} & *
    \end{array}
\right),
\]
hence $\ker \mathbf H$ is generated by the vector $\displaystyle \bigl( \psi_{d-2,1,0}, \, {-(d-1)} \psi_{d-1,0,0}, \, 0 \bigr)$.
Since $\lambda$ is the intersection multiplicity of $(W_1=0)$ with $\overline{X}_P$, we have $\varphi_{d-\lambda,0} \neq 0$, thus by rescaling $f_P$ if necessary, we may assume $\varphi_{d-\lambda, 0 } = 1$. Then by \eqref{eq: coefficient comparison of f(r(t))=0},
\begin{equation}
\begin{split}
\psi_{d-1,0,0} & = - \lambda! ( \gamma_2^{(\lambda)} )^{-1}, \\
\psi_{d-2,1,0} & = \lambda! (  \gamma_2^{(\lambda)} )^{-1} \biggl( -\varphi_{d-\lambda-1, 1} -  \frac{\lambda \gamma_1''}{2!} + \frac{ \gamma_2^{(\lambda+1)}}{(\lambda+1)  \gamma_2^{(\lambda)} } \biggr).
\end{split}
\end{equation}
It is easy to see that $\ker \mathbf H$ is generated by the point
\[
    \bar T_P := \left[-\varphi_{d-\lambda-1,1} -  \frac{\lambda \gamma_1''}{2!} + \frac{ \gamma_2^{(\lambda+1)}}{(\lambda+1)  \gamma_2^{(\lambda)} }  ,\ d-1,\ 0 \right] \in \nP^2,
\]
so the corresponding lift $T_P \in \mathbf{T}_x X$ is
\begin{equation}\label{eq: kernel of Hessian, osculating/inflectionary case}
    T_P = \left[-\varphi_{d-\lambda-1,1} -  \frac{\lambda \gamma_1''}{2!} + \frac{ \gamma_2^{(\lambda+1)}}{(\lambda+1) \gamma_2^{(\lambda)} } ,\ d-1,\ 0,\ 0 \right]. \qedhere
\end{equation}
\end{proof}

We shall interpret the number $\varphi_{d-\lambda - 1, 1, 0}$ in \eqref{eq: kernel of Hessian, osculating/inflectionary case} as follows.
\begin{enumerate}
\item Let $\bar {\mathcal R} = ( \overline{X}_P \cap \mathbf{T}_{\bar x_P} \overline{X}_P ) \setminus (W_1 = 0)$. For each $\bar R_i \in \bar{\mathcal R}$, write $\bar R_i = [ \bar \alpha_i, 1, 0]$. Then
\[ 
\varphi_{d-\lambda - 1, 1, 0} = - \sum_{\bar R_i \in \bar{\mathcal R}} { \bar \alpha_i}. 
\]
\end{enumerate}
Note that the above condition is equivalent to the following.
\begin{enumerate}\setcounter{enumi}{1}
\item Let $\Xi = \pi_{P,\nP^3}^* \mathbf{T}_{\bar x_P} \overline{X}_P$, and let $\mathcal R = ( X \cap \Xi )  \setminus ( Z_1 = 0)$. For each $R_i \in \mathcal R$, let $P_i$ be the intersection point of the line $\langle P , R_i \rangle$ with $\mathbf{T}_x X$. If $P_i = [\alpha_i, 1, 0,0]$, then $\varphi_{d-\lambda-1,1,0} = - \sum_{R_i \in \mathcal R} \alpha_i$.
\end{enumerate}

\medskip\par
In the following lemma, we will prove that the value $\varphi_{d-\lambda-1,1,0}$ varies along the line $L = \langle x, P \rangle$ unless the set-theoretic intersection $X \cap \Xi$ is contained in $\nT_x X$. The exceptional case may occur (see Example \ref{ex:curvecase} $(2)$).
\begin{lemma}\label{lem: moving lemma}
Let $\Xi = \nP^2$ be a projective plane with a homogeneous coordinates $[\xi_0,\xi_1,\xi_2]$, and let $\mathbf T \subseteq \Xi$ be the line $(\xi_2=0)$. Let $\mathcal R := \{R_1,\dots,R_r\} \subseteq \Xi$ be a finite set of points satisfying the following properties.
\begin{enumerate}
    \item $\mathcal R \not\subseteq \mathbf T$;
    \item $\mathcal R \cap (\xi_1=0) = \emptyset$.
\end{enumerate}
Let $P \in \Xi$ be a general point and for each $i=1,\dots,r$, let $P_i = [ \alpha_i(P), 1, 0]$ be the intersection of the line $\langle P, R_i \rangle$ with $\mathbf T$. Then, $\sum_{i=1}^r \alpha_i(P)$ varies over $P$.
\end{lemma}
\begin{proof}
Assume the contrary that $P \mapsto \sum_{i=1}^r \alpha_i(P)$ is constant on an open subset of $\Xi$. Let us write $R_i = [ \zeta_i, 1, \eta_i]$ and $P = [\zeta, 1, \eta]$ for $\eta \not\in \{ \eta_1,\dots,\eta_r\}$. The line $\langle P, R_i \rangle$ intersects $L$ at
\[
    \left[ \zeta_i \eta - \eta_i \zeta ,\ \eta - \eta_i ,\ 0 \right],
\]
hence $\displaystyle \alpha_i(P) = \frac{  \zeta_i \eta - \eta_i \zeta }{\eta - \eta_i}$. Consider the index set $\mathcal I = \{ i : \eta_i \neq 0\} \subseteq \{1,\dots, r \}$, which is nonempty by the assumption (1). We may write
\[
    \sum_{i=1}^r \alpha_i(P) = \sum_{i \not\in \mathcal I} \zeta_i + \sum_{i \in \mathcal I} \frac{ \zeta_i \eta - \eta_i \zeta }{ \eta - \eta_i}.
\]
Let $F(\zeta,\eta)$ be the (homogeneous) polynomial
\[
    \prod_{i \in \mathcal I } (\eta - \eta_i) \sum_{i \in \mathcal I } \frac{\zeta_i \eta - \eta_i \zeta}{ \eta - \eta_i}
\]
of degree $\lvert \mathcal I \rvert>0$. Our assumption says that $F(\zeta,\eta) = 0$ for general $[\zeta,\eta] \in \nP^1$. This implies that $F = 0$, but $F(1,0) = - \lvert \mathcal I \rvert \neq 0$, hence a contradiction.
\end{proof}

\begin{corollary}\label{cor: moving lemma}
Notations as in Proposition~\ref{prop: coefficient comparison for projection from P, general case}. Assume that the plane $\Xi := \langle L , \mathbf T_x X \rangle = (Z_2=0)$ contains a point in $X \setminus \mathbf T_x X$. Then, the map $P \mapsto T_P$, where $P$ varies over birational centers of $(X,x)$ in $L$, is not constant.
\end{corollary}
\begin{proof}
By \eqref{cor: kernel of Hessian, osculating or inflectionary case}, $T_P$ is determined by $\varphi_{d-\lambda-1, 1, 0},\, d,\, \lambda,\, \gamma_1'',\, \gamma_2^{(\lambda)},\, \gamma_2^{(\lambda+1)}$. Applying Lemma~\ref{lem: moving lemma} to $\mathcal R = X \cap \Xi$ and $\mathbf T = \mathbf T_x X$, we see that $\varphi_{d-\lambda-1,1,0}$ varies over $L$, while the others are invariant.
\end{proof}
\begin{corollary}\label{cor: Vout separates tangnets, inflectionary case}
If $x \in X$ is an inflection point, then the assumption \eqref{eq: contrary assumption} fails to hold and the right hand side of \eqref{eq: separated tangents 2} holds.
\end{corollary}
\begin{proof}
We show that the assumption \eqref{eq: contrary assumption} fails if $x \in X$ is an inflection point. Let $\Xi$ be a general plane containing $\mathbf T_x X$. Then the set-theoretic intersection $X \cap \Xi$ is not contained in $\mathbf T_x X$. Let $P \in \Xi$ be a birational center, and let $L = \langle x, P \rangle$. By Corollary~\ref{cor: moving lemma}, there exists a birational center $Q \in L$ such that $T_P \neq T_Q$. Let $R$ be a general point in $\nP^3$ such that both $\{P,R\}$ and $\{Q,R\}$ are admissible birational centers of $(X,x)$. If the assumption \eqref{eq: contrary assumption} holds, then $\mathcal L_{R,T_P}$ contains $L$, thus $Q \in \mathcal L_{R,T_P}$. This is absurd since $R \not\in \mathcal L_{Q,T_P} = \langle Q, \mathbf T_x X \rangle$ by Lemma~\ref{lem: when L_PT is not a pencil}.
\end{proof}
The remaining cases are considered in the following proposition, finishing the proof of Theorem~\ref{thm:D_outva}~(2).

\begin{proposition}\label{prop: Vout separates tangnets, non-inflectionary case}
If $x \in X$ is not an inflection point, then the assumption \eqref{eq: contrary assumption} fails to hold and the right hand side of \eqref{eq: separated tangents 2} holds.
\end{proposition}
\begin{proof}
Let $\Theta$ be the plane osculating to $X$ at $x$, and let $R,S \in \nP^3 \setminus \Theta$ be an admissible pair of birational centers such that $S \in \mathcal L_{R,T_0}$ for some $T_0 \in \mathbf T_x X$. By Proposition~\ref{prop: union of L_PT cap L_QT}, we have the plane $\Theta(R,S)$ which fits into
\[
    \bigcup_{T \in \mathbf T_x X} ( \mathcal L_{R,T} \cap \mathcal L_{S,T} ) = \langle x, R , S \rangle \cup \Theta(R,S).
\]
If $\Theta \neq \Theta(R,S)$, then we are done by Proposition~\ref{prop: V_out separates tangents, non-osculating case}.

\medskip\par
Assume that $\Theta = \Theta(R,S)$ for any admissible pair $R, S \in \nP^3 \setminus \Theta$. For $T_0 \in \mathbf T_x X \setminus \{x \}$, let $\mathcal M_{T_0}$ be the set of birational centers $Q$ for which $\mathcal L_{Q,T_0} = \nP^3$. Suppose that $P, Q \in \mathcal M_{T_0}$ are two distinct birational centers. Since $\mathcal L_{P,T_0} = \nP^3 = \mathcal L_{Q,T_0}$, by the assumption \eqref{eq: contrary assumption} we have $P, Q \subseteq \mathcal L_{R,T_0}$ for a general point $R \in \nP^3$. Since $\mathcal L_{R,T_0} \neq \nP^3$ and $\langle x, P, Q, R \rangle \subseteq \mathcal L_{R,T_0}$, we see that $\langle x, P, Q, R \rangle$ is a plane, and $x, P, Q$ are collinear. We have $Q \in \mathcal M_{T_Q}$ by Proposition~\ref{prop: inflection points and L_PT}. If we assume that the birational center $P$ does not belong to $\mathcal M_{T_Q}$, then we have $T_P \neq T_Q$, and the argument in the proof of Corollary~\ref{cor: Vout separates tangnets, inflectionary case} works in the same way, hence we are done.

\medskip\par
Now, it remains to consider the case: $\overline{ \mathcal M}_{T_P} = \langle x, P \rangle$ for each birational center $P \in \Theta$. Fix a birational center $P \in \Theta$ and consider a general point $R\in \nP^3$. Let $S \in \langle x, R \rangle$ be a general point. Then, both $\mathcal L_{R,T_P}$ and $\mathcal L_{S,T_P}$ contain $P$, and thus, 
\begin{equation}\label{eq: worst case arrangement}
 \mathcal L_{R, T_P} =  \langle x, P, R \rangle = \langle x, P, S \rangle = \mathcal L_{S,T_P}.   
\end{equation}
Let $L = \langle R, S \rangle$ be the line. We will apply Proposition~\ref{prop: coefficient comparison for projection from P, general case} with respect to $L$ and the projection $\pi_R$ from $R$. As in Proposition~\ref{prop: coefficient comparison for projection from P, general case}, we set homogeneous coordinates $[Z_0,Z_1,Z_2,Z_3]$ of $\nP^3$ so that $x = [1,0,0,0]$, $L = (Z_1=Z_2=0)$, and $\mathbf T_x X = (Z_2=Z_3=0)$. Let $f_R \in \nC[W_0,W_1,W_2]$ be the defining equation of $\overline{X}_R := \pi_{R,X}(X)$, where $\pi_{R,X}$ maps $[ Z_0, Z_1, Z_2, Z_3] $ to $ [ r_3 Z_0 - r_0Z_3, Z_1, Z_2 ] $. If we write
\[
    f_R(W_0,W_1,W_2) = W_1^2 \varphi(W_0,W_1) + W_2 \psi(W_0, W_1, W_2),
\]
where $\varphi(W_0,W_1) = \sum_{i,j} \varphi_{ij} W_0^i W_1^j$ and $\psi(W_0,W_1,W_2) = \sum_{i,j,k} \psi_{ijk} W_0^i W_1^j W_2^k$, then for a local parametrization $\gamma \colon \Delta \to X$, we have
\begin{equation}
\left\{
\begin{array}{l}
    \displaystyle \varphi_{d-2,0} + \psi_{d-1,0,0} \frac{ \gamma_2''}{2!}  = 0, \\
    \displaystyle \gamma_1'' \varphi_{d-2,0} + \varphi_{d-3,1} + \frac{ \gamma_2''}{2!} \psi_{d-2,1,0} +  \frac{ \gamma_2''' }{3!} \psi_{d-1,0,0}  = 0.
\end{array}\right.
\end{equation}
Since $R$ is chosen outside the osculating plane $\Theta$, we see that $\bar x_R \in \overline{X}_R$ is not an inflection point, hence $\varphi_{d-2,0} \neq 0$. Assume $\varphi_{d-2,0} = 1$. Then, the Hessian matrix of $f_R$ evaluated at $\bar x_R = [1,0,0]$ is
\[
\left (
\begin{array}{ccc}
0 & 0 & (d-1) \psi_{d-1,0,0} \\
0 & 2 & \psi_{d-2,1,0} \\
(d-1) \psi_{d-1,0,0} & \psi_{d-2,1,0)} & *
\end{array}
\right) .
\]
If we write $R = [r_0,0,0,r_3]$ and $T_P = [t_0, 1, 0 ,0 ]$, then $\mathcal L_{R,T_P}$ is the set of points $[q_0,q_1,q_2,q_3] \in \nP^3$ defined by the following equation\,(see \eqref{eq: H thru T, general form} and the discussions therein).
\[
    \Biggl( \biggl( t_0 \frac{\pd}{\pd Z_0} + \frac{\pd}{\pd Z_1} \biggr) %
    \pi_{R,\nP^3}^* \biggl( (r_3q_0-r_0q_3) \frac{\pd f_R}{\pd W_0} +
    q_1 \frac{\pd f_R}{\pd W_1} +
    q_2 \frac{\pd f_R}{\pd W_2}
    \biggr)\Biggr) (x)=0.
\]
Using chain rule, one can show
\begin{align*}
\left( \frac{\pd}{\pd Z_0} \right) \circ \pi_{R,\nP^3}^* &= r_3 \left( \pi_{R,\nP^3}^* \circ \frac{\pd}{\pd W_0}\right),
\\ \left( \frac{\pd}{\pd Z_1} \right) \circ \pi_{R,\nP^3}^* &= \pi_{R,\nP^3}^* \circ \frac{\pd}{\pd W_1}.
% \\ \left( \frac{\pd}{\pd Z_2} \right) \circ \pi_R^* &= \pi_R^* \circ \frac{\pd}{\pd W_2}
% \\ \left( \frac{\pd}{\pd Z_3} \right) \circ \pi_R^* &= -r_0 \left( \pi_R^* \circ \frac{\pd}{\pd W_0}\right)
\end{align*}
Hence, the equation of $\mathcal L_{R,T_P}$ is
\[
2q_1 + ( t_0 r_3 (d-1) \psi_{d-1,0,0} + \psi_{d-2,1,0}) q_2 = 0.
% Author's note:
% It seems like this equation is inconsistent with respect to the choice of ideal (r_3 W_0 - r_0W_3 , W_1, W_2). However, one needs to consider the variation of \bar f_R (and \varphi_ij, \psi_ijk accordingly) after the change of coordinates
% For instance, if we make the change [W_0,W_1,W_2] \mapsto [2W_0, W_1, W_2] then the equation \bar f_R has to be revised by \bar f_r' := 2^{d-2} \bar f_R( 1/2 W_0, W_1, W_2 ). Here, the coefficient 2^{d-2} appears due to the condition \phi_{d-2,0,0}=1. It can be shown that the corresponding constant \psi_ijk are rescaled suitably, so that the equation of L_{R,T_P} is consistent even if the value of r_3 is doubled.
\]
By Proposition~\ref{prop: coefficient comparison for projection from P, general case}, we have $\psi_{d-1,0,0} = - 2 (\gamma_2'')^{-1}$ and
\[
    \psi_{d-2,1,0} = 2 (\gamma_2'')^{-1}\left(  -\varphi_{d-3,1} - \gamma_1'' + \frac{\gamma_2'''}{3 \gamma_2''}\right).
\]
Suppose $R = [r_0, 0,0,1 ]$. Then the values of $\gamma_1''$, $\gamma_2''$, $\gamma_2'''$ are independent of $r_0$. On the other hand, we may assume that the plane $\Xi := \langle L, \mathbf T_x X \rangle$ satisfies the following: the set-theoretic intersection $X \cap \Xi$ is not contained in $\mathbf T_x X$. Then by Lemma~\ref{lem: moving lemma}, the value $\varphi_{d-3,1}$ depends on $r_0$. It follows that the rational map
\[
    L \dashrightarrow (\nP^3)^\vee,\quad R \longmapsto \mathcal L_{R,T_P}
\]
is not constant. This is a contradiction to \eqref{eq: worst case arrangement}.
\end{proof}

The following are examples showing that subtle cases about inflection points carefully treated in this section actually exist.

\begin{example}\label{ex:curvecase}
$(1)$ Let $C$ be a smooth projective curve of genus $2$, and $f \colon C \to \nP^1$ be the hyperelliptic fibration ramified in $x,y,z \in C$. The line bundle $L:=\omega_C(3x)$ gives an embedding $C \subseteq \nP^3$. As $L(-x)=\omega_C^2 = f^* \sO_{\nP^1}(2)$, the inner projection $\pi_x \colon C \to \overline{C} \subseteq \nP^2$ coincides with the hyperelliptic fibration. Thus $x \in \sC(C)$. For any divisor $H \in |L|$, if $\mult_x H \geq 2$, then $H-2x = \omega_C(x)$ so that $\mult_x H \geq 3$. This means that $x$ is an inflection point of $C \subseteq \nP^3$.\\[3pt]
$(2)$ Let $C \subseteq \nP^3$ be a twisted cubic curve. For any point $x \in C$, the hyperplane $\Theta:=\langle 3x \rangle$ is an osculating plane to $C$ at $x$. Then $\Theta \cap C = \{x\}$ as a set.
\end{example}

%%%%%%%%%%%%%%%%%%%%%%%%%%%%%%%%%%%%%%%%%%%%%%%%%%%%%%%
\section{Base point freeness of $\lvert D_{\inn} \rvert$}\label{sec:bpf}
%%%%%%%%%%%%%%%%%%%%%%%%%%%%%%%%%%%%%%%%%%%%%%%%%%%%%%%

In this section, we investigate when $D_{\inn}$ is base point free, and prove Theorem \ref{thm:bpf}.

\subsection{The curve case}
If $X \subseteq \nP^r$ is not a Noma's exceptional variety and $D_{\inn}|_C$ is not base point free for a general curve section $C \subseteq \nP^{e+1}$, then $D_{\inn}$ is not base point free. We first consider the curve case in order to find examples where $D_{\inn}$ is not base point free.

\begin{proposition}
Let $C \subseteq \nP^{e+1}$ be a non-degenerate smooth projective curve of genus $g$, codimensnion $e \geq 2$, and degree $d$, and $H$ be its hyperplane section. Then 
$-K_C + (d-e-2)H$ is not base point free if and only if $C \subseteq \nP^{e+1}$ is a rational normal curve or $C \subseteq \nP^3$ is linearly normal with $g \geq 2$ and $C$ is a birational divisor of type $(m, \sO_{\nP^1}(1))$ for $m \geq 2$ on a conical scroll $\overline{F}$ over $\nP^1$ which is obtained by the birational embedding of $F:=\nP(\sO_{\nP^1} \oplus \sO_{\nP^1}(2))$ by $\lvert \sO_F(1) \rvert$.
\end{proposition}

\begin{proof}
We put $D:=-K_C + (d-e-2)H$. As the assertion is easy to check when $g=0$ or $1$, we assume that $g \geq 2$ and $d \geq e+3$. Suppose that $d=e+3$. By Riemann--Roch theorem and Clifford theorem,  $g=2$ and $C \subseteq \nP^{e+1}$ is linearly normal. If $e \geq 3$, then $\deg D = e+1 \geq 4=2g$ so that $D$ is base point free. If $e=2$, then $\deg D=3$. Thus $D$ is not base point free if and only if $D=K_C + x$ (equivalently, $H=2K+x$) for $x \in C$. 
In this case, $C$ is a birational divisor of type $(2, \sO_{\nP^1}(1))$ on a conical scroll $\overline{F}$ over $\nP^1$ which is obtained by the birational embedding of $F:=\nP(\sO_{\nP^1} \oplus \sO_{\nP^1}(2))$ by $\lvert \sO_F(1) \rvert$.

\medskip

Suppose that $d \geq e+4$. To show that $D$ is base point free, it suffices to check that
$$
-2g+2+(d-e-2)d=\deg D \geq 2g,
$$
which can be restated as
\begin{equation}\label{eq:D_innbpfcurve}
\frac{(d-e-2)d+2}{2} \geq 2g
\end{equation}
By Castelnuovo's genus bound (\ref{eq:Castelnuovo-big}), we have
$$
\frac{(d-e-2)d + (\varepsilon+1)(e+1-\varepsilon)}{e} \geq 2g,
$$
where $m:=\left\lfloor (d-1)/e \right\rfloor$ and $\varepsilon := (d-1) - me$ with $m \geq 1$ and $0 \leq \varepsilon \leq e-1$. Thus (\ref{eq:D_innbpfcurve}) can be deduced from
$$
\frac{(d-e-2)d+2}{2} \geq \frac{(d-e-2)d + (\varepsilon+1)(e+1-\varepsilon)}{e}
$$
which is equivalent to
\begin{equation}\label{eq:(e-2)(d-e-2)d+2e}
(e-2)(d-e-2)d+2e \geq 2(\varepsilon+1)(e+1-\varepsilon).
\end{equation}
First, consider the case of $e \geq 3$. 
As $d \geq e+4$, we obtain
$$
(e-2)(d-e-2)d+2e \geq \displaystyle  2(e-2)(e+4)+2e  \geq 2 \left(\frac{e+2}{2} \right)^2 \geq 2 (\varepsilon+1)(e+1-\varepsilon).
$$
Next, consider the case of $e=2$ (in this case, (\ref{eq:(e-2)(d-e-2)d+2e}) does not hold). Keep in mind that $d \geq 6$. Castelnuovo's genus bound (\ref{eq:Castelnuovo-big}) says
$$
\frac{(d-4)d+(\varepsilon+1)(3-\varepsilon)}{2} \geq 2g.
$$
If the strict inequality holds, then (\ref{eq:D_innbpfcurve}) holds so that $D$ is base point free. Suppose that the equality holds. If $d=2m$ is even, then $C \subseteq \nP^3$ is a complete intersection of type $(2,m)$ (see e.g., \cite[page 119]{ACGH}). In this case, $D=K_C$ is very ample since $d \geq 6$. Suppose that $d=2m+1$ is odd ($\varepsilon=0$). As $((d-4)d+3)/2 \geq 2g$, we have $\deg D\geq 2g-1$. Thus $D$ is not base point free if and only if $D=K_C+x$ (equivalently, $2(K_C - (m-2)H) = H-x$) for $x \in C$. In this case, the equality holds in Castelnuovo's genus bound, so $C \subseteq \nP^3$ is projectively normal and $\lvert K_C - (m-2)H \rvert$ gives a branch covering $f \colon C \to \nP^1$ of degree $m$ (\cite[Corollary 2.6]{ACGH}). Since $H-x = 2(K_C-(m-2)H)$ and $h^0(C, H-x)=3$, it follows that $\lvert H-x \rvert$ gives a morphism 
$$
C \xrightarrow{~f~} \nP^1 \xrightarrow{~v_2~} v_2(\nP^1) \subseteq \nP^2.
$$
This means that $x \in \sC(X)$ and $C$ is a birational divisor of type $(m, \sO_{\nP^1}(1))$ for $m \geq 3$ on a conical scroll $\overline{F}$ over $\nP^1$ which is obtained by the birational embedding of $F:=\nP(\sO_{\nP^1} \oplus \sO_{\nP^1}(2))$ by $\lvert \sO_F(1) \rvert$. It is easy to check that in this case $D$ is not base point free.
\end{proof}

The only case where a variety is not a Noma's exceptional variety and its general curve section appears in the above proposition is the following Noma's example in \cite{N3}.

\begin{example}[{Noma}]\label{ex:Nomaexcsurf}
Let $\overline{S} \subseteq \nP^4$ be a conical rational normal scroll with vertex $p$, which is obtained by the birational embedding $\psi \colon S \to \overline{S}$ of $S:=\nP(\sO_{\nP^1} \oplus \sO_{\nP^1}(1)^{\oplus 2})$ given by $\lvert \sO_S(1) \rvert$. Denote by $\pi \colon S \to \nP^1$ the canonical projection, and take a general member $X \in \lvert \sO_S(b) \otimes \pi^*\sO_{\nP^1}(1) \rvert$ for $b \geq 2$. Then $\psi|_X$ gives an embedding of the smooth projective surface $X$ into $\nP^4$. Note that $X \subseteq \nP^4$ is not a Noma's exceptional variety and $\deg X = 2b+1$. Noma \cite[Theorem 0.3]{N3} proved that $\Bs\lvert D_{\inn} \rvert=\sC(X)=\{p \}$. Since $\length \langle p, u \rangle = b$ for general $u \in X$, it follows that 
 $\overline{\pi_{p, X}(X \setminus \{p\})} = Q^2 \subseteq \nP^3$ is a quadric surface. Actually, every non-degenerate smooth projective surface $X \subseteq \nP^4$ such that $\overline{\pi_{p, X}(X \setminus \{p\})} = Q^2 \subseteq \nP^3$ is a quadric surface for $p \in \sC(X)$ arises in this way, and $D_{\inn}$ is not base point free.
\end{example}

\subsection{General framework}
Now, we show Theorem \ref{thm:bpf}. We assume that $n\geq 2$ and $X \subseteq \nP^r$ is not a Noma's exceptional variety with $\sC(X)$ a nonempty finite set. For $x \in \sC(X)$, consider the inner projection $\pi_{x, X} \colon X \dashrightarrow Y \subseteq \nP^{r-1}$ centered at $x$. Suppose that $Y$ is smooth. 
We recall necessary notations in Subsection \ref{subsec:birdiv}. Denote $H_Y$ hyperplane sections of $Y \subseteq \nP^{r-1}$. We have $d=\mu d_Y+1$ for $\mu=\deg \pi_{x,X} \geq 2$, where $d_Y:= \deg Y$. There is an effective divisor $L$ on $Y$ such that $L \subseteq \nP^{r-1}$ is a linear subspace. The cone $\overline{F} \subseteq \nP^r$ over $Y$ with the vertex $x$ is the image of the birational map $\sigma_F \colon F \to \overline{F}$ with the exceptional divisor $\Gamma = H_F - \tau_F^* H_Y$, where $F:=\nP(\sO_Y \oplus \sO_Y(1))$ with the tautological divisor $H_F$ and the canonical projection $\tau_F \colon F \to Y$. Note that 
$\sigma_F^* \sO_{\overline{F}}(1) = \sO_F(H_F)$ and $\Gamma \cong Y$. There is a prime divisor $\widetilde{X} \in |\mu H_F + \tau_F^* L|$ such that $X=\sigma_F(\widetilde{X})$ and $\sigma:=\sigma_F|_{\widetilde{X}} \colon \widetilde{X} \to X$ is the blow-up of $X$ at $x$ with the exceptional divisor $E=\Gamma|_{\widetilde{X}}$. Put $\widetilde{H}=\sigma^* H = H_F|_{\widetilde{X}}$ and $\tau:=\tau_F|_{\widetilde{X}}$. Note that $\tau|_E \colon E \to L$ is an isomorphism. We have a commutative diagram
$$
\xymatrix{
\widetilde{X} \ar[r]^-{\sigma} \ar[rd]_-{\tau} & X \ar@{-->}[d]^-{\pi_{x,X}} \rlap{~$\subseteq  \nP^r$} \\
& Y  \rlap{~$\subseteq \nP^{r-1}$.}
}
$$

\begin{lemma}
If $D_F:=(d-e-\mu)H_F - \tau_F^* (K_Y + nH_Y + L)$ is a divisor on $F$, then $\sigma^* D_{\inn} = D_F|_{\widetilde{X}}$.
\end{lemma}

\begin{proof}
As $K_F = -2H_F + \tau_F^* (K_Y + H_Y)$, we have
$$
K_{\widetilde{X}}=(K_F + \mu H_F + \tau_F^* L)|_{\widetilde{X}} = (\mu - 2) \widetilde{H} + \tau^*(K_Y + H_Y + L).
$$
Note that $\sigma^* K_X = K_{\widetilde{X}} - (n-1)E$. Then
$$
\sigma^* D_{\inn} = (d-n-e-1) \widetilde{H} - K_{\widetilde{X}} + (n-1)E = (d-e-\mu)\widetilde{H} - \tau^* (K_Y + nH_Y + L). \qedhere
$$  
\end{proof}

Since $H_F|_{\Gamma} = 0$, it follows that 
$$
D_F|_{\Gamma} = -K_Y - nH_Y - L~~\text{ and }~~(D_F - \widetilde{X})|_{\Gamma} = -K_Y - nH_Y - 2L.
$$
Note that $(-K_Y - nH_Y - L)|_L = -K_L - nH_Y|_L = 0$ as $L=E=\nP^{n-1}$.
For any integer $m \geq 0$, we have a commutative diagram with short exact sequences
$$
\xymatrixrowsep{0.35in}
\xymatrixcolsep{0.25in}
\xymatrix{
& 0 \ar[d] & 0 \ar[d] & 0 \ar[d] &\\
0 \ar[r] & \sO_F(D_F - \widetilde{X} - \Gamma + mH_F) \ar[r]^-{\cdot \widetilde{X}} \ar[d]^-{\cdot \Gamma} & \sO_F(D_F - \Gamma + mH_F) \ar[r] \ar[d]^-{\cdot \Gamma} & \sO_{\widetilde{X}}(\sigma^* D_{\inn} - E + mH) \ar[r] \ar[d]^-{\cdot E} & 0\\
0 \ar[r] & \sO_F(D_F - \widetilde{X}  + mH_F) \ar[r]^-{\cdot \widetilde{X}} \ar[d] & \sO_F(D_F + mH_F) \ar[r] \ar[d] & \sO_{\widetilde{X}}(\sigma^* D_{\inn}  + mH) \ar[r] \ar[d] & 0\\
0 \ar[r] & \sO_Y(-K_Y - nH_Y - 2L) \ar[r]^-{\cdot L} \ar[d] & \sO_Y(-K_Y - nH_Y - L) \ar[r] \ar[d] & \sO_{E} \ar[r] \ar[d] & 0\\
& 0 & 0 & ~0. &
}
$$
Notice that $x \not\in \Bs |D_{\inn} + mH|$ if and only if the restriction map 
$$
r_{\widetilde{X}}^m \colon H^0(\widetilde{X}, \sO_{\widetilde{X}}(\sigma^* D_{\inn} + m \widetilde{H})) \longrightarrow H^0(E, \sO_E)
$$
is surjective. We know that $r_{\widetilde{X}}^m$ is surjective whenever $m \gg 0$. 
We want to show that this happens when $m = 0$. To this end, we make two observations.

\begin{lemma}\label{lem:observationforbps}
$(1)$ $H^0(Y, \sO_Y(-K_Y - nH_Y - L))=0$.\\[3pt]
$(2)$ If $d-e-2\mu-1 \geq -1$, then we have a splitting short exact sequence
$$
0 \longrightarrow \tau_{F,*} \sO_F(D_F - \widetilde{X} - \Gamma + mH_F) \longrightarrow \tau_{F, *} \sO_F( D_F - \widetilde{X} + m H_F) \longrightarrow \sO_Y(-K_Y - nH_Y - 2L) \longrightarrow 0.
$$
\end{lemma}

\begin{proof}
As $K_Y + nH_Y$ is base point free by \cite[Lemma 2]{E}, the assertion (1) holds. 
To show the assertion (2), we apply $\tau_{F,*}$ to the left-most vertical short exact sequence in the above commutative diagram. We write
$$
\begin{array}{l}
D_F - \widetilde{X} - \Gamma + mH_F = (d-e-2\mu -1+m) H_F - \tau_F^* (K_Y + (n-1)H_Y+2L)\\
D_F - \widetilde{X} + mH_F = (d-e-2\mu+m)H_F - \tau_F^* (K_Y + nH_Y +2L).
\end{array}
$$
When $d-e-2\mu-1 \geq -1$, we have\\[-25pt]

\begin{equation}\label{eq:tauDF-X+mH}
\begin{footnotesize}
\begin{array}{l}
\tau_{F,*} \sO_F(D_F - \widetilde{X} - \Gamma + mH_F) = (\sO_Y(d-e-2\mu+m) \oplus \cdots \oplus \sO_Y(1)) \otimes \sO_Y(-(K_Y + nH_Y + 2L))\\
\tau_{F,*} \sO_F(D_F - \widetilde{X} + mH_F) = (\sO_Y(d-e-2\mu+m) \oplus \cdots \oplus \sO_Y(1) \oplus \sO_Y) \otimes \sO_Y(-(K_Y + nH_Y + 2L))
\end{array}
\end{footnotesize}
\end{equation}

\noindent and $R^1 \tau_{F,*} \sO_F(D_F - \widetilde{X} - \Gamma + mH_F) = 0$. 
\end{proof}

\begin{lemma}
If $d-e-2\mu-1 < -1$, then $n=e=2$ and $Y=Q^2 \subseteq \nP^3$ is a quadric surface.
\end{lemma}

\begin{proof}
Notice that
$$
-1 > d-e-2\mu-1 = \mu d_Y - e - 2\mu \geq \mu e - e - 2\mu = (\mu-1)(e-2) -2.
$$
Thus $e=2$, and $Y=Q^n \subseteq \nP^{n+1}$. If $n \geq 3$, then $\Pic(Y)$ is generated by $\sO_Y(1)$ so that $Y$ cannot contains a linear subspace $L$ of $\nP^{n+1}$ as a divisor. Hence $n=2$. 
\end{proof}

The case of $d-e-2\mu-1 < -1$ considered in the above lemma corresponds to the exceptional case in Theorem \ref{thm:bpf}. In this case, $D_{\inn}$ is not base point free (see Example \ref{ex:Nomaexcsurf}). In the remaining, we only need to focus on the case of $d - e - 2\mu -1 \geq -1$. We also assume that $Y \subseteq \nP^{r-1}$ is not a quadric hypersurface. By Lemma \ref{lem:observationforbps}, we obtain a commutative diagram
$$
\xymatrix{
H^0(\widetilde{X}, \sO_{\widetilde{X}}(\sigma^* D_{\inn} + m\widetilde{H})) \ar[r]^-{\delta_{\widetilde{X}}^m} \ar[d]_-{r_{\widetilde{X}}^m} & H^1(F, \sO_F(D_F - \widetilde{X} + mH_F)) \ar[r]^-{\eta^m} \ar@{->>}[d]_-{r_F^m} & H^1(F, \sO_F(D_F + mH_F)) \\
H^0(E, \sO_E) \ar@{^{(}->}[r]_-{\delta_E} & H^1(Y, \sO_Y(-K_Y - nH_Y - 2L)), &
}
$$
where $\delta_E$ is injective and $r_F^m$ is surjective.
Note that $r_{\widetilde{X}}^m$ is surjective if and only if there is $s_m \in H^1(F, \sO_F(D_F - \widetilde{X} + mH_F))$ such that $r_F^m(s_m) = \delta_E(1)$ and $\eta^m(s_m) = 0$. We know that this holds whenever $m \gg 0$. Recall that $D_F - \widetilde{X} = (d-e-\mu)H_F - \tau_F^* (K_Y + nH_Y + L)$. Suppose that
\begin{equation}\label{eq:H^1((d-e-2mu+m)H_Y - (K_Y + nH_Y + L))}
H^1(Y, \sO_Y((d-e-2\mu+m)H_Y - (K_Y + nH_Y + L)))=0~~\text{ for $m \geq 1$}.
\end{equation}
If this holds, then (\ref{eq:tauDF-X+mH}) yields
$$
H^1(F, \sO_F(D_F - \widetilde{X})) = H^1(F, \sO_F(D_F-\widetilde{X}+mH_F)), ~r_F^0=r_F^m,~\eta^0 = \eta^m~~\text{ for any $m \geq 0$}.
$$
Taking $s_0:=s_m \in H^1(F, \sO_F(D_F - \widetilde{X}))$ for $m \gg 0$, we obtain $r_F^0(s_0)=\delta_E(1)$ and $\eta^0(s_0)=0$. This implies that $r_{\widetilde{X}}^0$ is surjective as desired. 

\subsection{Proof of (\ref{eq:H^1((d-e-2mu+m)H_Y - (K_Y + nH_Y + L))})}
We verify (\ref{eq:H^1((d-e-2mu+m)H_Y - (K_Y + nH_Y + L))}) except for one case. In the exceptional case, we instead show
$$
H^1(\widetilde{X}, \sO_{\widetilde{X}}(\sigma^*D_{\inn} - E))=0,
$$
which directly implies that $r_{\widetilde{X}}^0$ is surjective. First, we prove the following elemenatry lemma.

\begin{lemma}\label{lem:H-Lnef}
Let $Y \subseteq \nP^{r-1}$ be a non-degenerate smooth projective variety of dimension $n$, and $L$ be an effective divisor on $Y$ such that $L \subseteq \nP^{r-1}$ is a linear subspace of dimension $n-1$. Then $\sO_L(L|_L) = \sO_{\nP^{n-1}}(\ell)$ for some integer $\ell \leq 0$, and $H_Y - L$ is nef, where $H_Y$ is a hyperplane section of $Y \subseteq \nP^{r-1}$.
\end{lemma}

\begin{proof}
Let $C \subseteq Y$ be an irreducible curve. Suppose that $C \not\subseteq L$. We can choose $H_Y \in |\sO_Y(1)|$ in such a way that $L \subseteq H_Y$ but $C \not\subseteq H_Y$. Then $(H_Y-L).C \geq 0$. Suppose that $C \subseteq L$. As $(H_Y-L).C = (H_Y-L)|_L .C$, it suffices to show that $\sO_L(L|_L)= \sO_{\nP^{n-1}}(\ell)$ for some $\ell \leq 0$, where $n:=\dim Y$. Recall from \cite[Lemma 2]{E} that $\lvert K_Y + nH_Y \rvert$ is base point free. Thus $\lvert (K_Y+nH_Y)|_L \rvert$ is also base point free, so we may write $\sO_L(K_Y|_L) = \sO_{\nP^{n-1}}(k)$ for some integer $k \geq -n$. Since $\sO_L((K_Y+L)|_L) = \sO_L(K_L) = \sO_{\nP^{n-1}}(-n)$, it follows that $\sO_{\nP^{n-1}}(\ell) = \sO_L(L|_L) = \sO_L(-n-k)$. Thus $\ell \leq 0$.
\end{proof}

For an integer $m \geq 1$, we write
$$
(d-e-2\mu+m)H_Y - (K_Y + nH_Y + L) = K_Y + (-2K_Y + (d-e-n-2\mu) H_Y) + (mH_Y - L).
$$
By Lemma \ref{lem:H-Lnef}, $mH_Y - L$ is nef. If $-2K_Y + (d-e-n-2\mu) H_Y$ is ample, then Kodaira vanishing implies (\ref{eq:H^1((d-e-2mu+m)H_Y - (K_Y + nH_Y + L))}).

\begin{lemma}
If $(\mu-2)(d_Y-2)+n \geq e$, then (\ref{eq:H^1((d-e-2mu+m)H_Y - (K_Y + nH_Y + L))}) holds.
\end{lemma}

\begin{proof}
The given condition is equivalent to
$$
d-e-n-2\mu = \mu d_Y - (e-1) - n - 2 \mu \geq 2d_Y - 2n - 3.
$$
As $-K_Y +(d_Y - n-2)H_Y$ is the double point divisor from outer projection for $Y \subseteq \nP^{r-1}$, it is always nef. Thus
$$
-2K_Y + (d-e-n-2\mu) H_Y = \underbrace{-2(K_Y + (d_Y - n-2)H_Y)}_{\text{nef}} + \underbrace{( (d-e-n-2\mu) - (2d_Y - 2n-4))}_{\geq 1} H_Y
$$
is ample.
\end{proof}

We henceforth assume that $(\mu-2)(d_Y-2)+n < e$. We must have $\mu=2$ since $d_Y \geq e$ and $n \geq 2$. Then we get $n<e$. 

\begin{lemma}
If $Y$ is not a Noma's exceptional variety, then (\ref{eq:H^1((d-e-2mu+m)H_Y - (K_Y + nH_Y + L))}) holds.
\end{lemma}

\begin{proof}
As $-K_Y + (d_Y - n-e)H_Y$  is the double point divisor from inner projection for $Y \subseteq \nP^{r-1}$, it is always nef. Note that
$$
d-e-n-2\mu = 2d_Y -n-e-3 \geq 2d_Y - 2n - 2e+1.
$$
Then
$$
-2K_Y + (d-e-n-2\mu) H_Y  = \underbrace{-2(K_Y + (d_Y - n-e)H_Y)}_{\text{nef}} + \underbrace{((d-e-n-2\mu)-(2d_Y -2n-2e))}_{\geq 1}H_Y
$$
is ample. 
\end{proof}

It only remains to consider the case that $\mu = 2$ and $Y \subseteq \nP^{r-1}$ is a Noma's exceptional variety. When $Y \subseteq \nP^{r-1}$ is the second Veronese surface, (\ref{eq:H^1((d-e-2mu+m)H_Y - (K_Y + nH_Y + L))}) is trivial. In the next two lemmas, we consider the cases when $Y \subseteq \nP^{r-1}$ is a scroll over a smooth projective curve or $Y \subseteq \nP^{r-1}$ is a Roth variety.

\begin{lemma}
If $\mu=2$ and $Y \subseteq \nP^r$ is a scroll over a smooth projective curve $C$ of genus $g \geq 0$, then (\ref{eq:H^1((d-e-2mu+m)H_Y - (K_Y + nH_Y + L))}) holds.
\end{lemma}

\begin{proof}
There is a very ample vector bundle $\sE$ on $C$ such that $Y=\nP(\sE)$ and $\sO_Y(1)=\sO_{\nP(\sE)}(1)$. Let $\pi_Y \colon Y \to C$ be the canonical projection. Suppose that $L$ is not a fiber of $\pi_Y$. Then $C=\nP^1$ and $\sE = \sO_{\nP^1}(1) \oplus \sO_{\nP^1}(a)$ for some integer $a \geq 2$. Moreover, $L=H_Y - aF_Y$, where $F_Y$ is a fiber of $\pi_Y$. 
One can easily check that
$$
(d-e-2\mu+m)H_Y - (K_Y + nH_Y + L) = (2d_Y - e-4+m)H_Y + F_Y
$$
is nef for $m \geq 1$. As $Y$ is of Fano type, Kawamata--Viehweg vanishing yields (\ref{eq:H^1((d-e-2mu+m)H_Y - (K_Y + nH_Y + L))}). Now, suppose that $L=\pi_Y^{-1}(z)$ is a fiber of $\pi_Y$ over $z \in C$. We write
$$
(d-e-2\mu+m)H_Y - (K_Y + nH_Y + L) = (2d_Y- e-3+m)H_Y  - \pi_Y^*(K_C + \det \sE+z).
$$
Here $\deg (K_C + \det \sE+z) = 2g-1 + d_Y$. By \cite[Lemmas 1.12 and 2.5]{Butler}, (\ref{eq:H^1((d-e-2mu+m)H_Y - (K_Y + nH_Y + L))}) holds if
$$
\mu^-(S^{2d_Y - e-3+m} \sE \otimes \sO_C(-K_C - \det \sE - z)) = (2d_Y - e-3+m)\mu^-(\sE) - 2g+1-d_Y > 2g-2
$$
for $m \geq 1$. We know that
$$
-K_Y + (d_Y-n-2)H_Y = (d_Y - 2)H_Y + \pi_Y^*(-K_C - \det \sE)
$$
is ample. By Miyaoka's criterion (see e.g. \cite[Lemma 5.4]{Butler}), this implies that
$$
(d_Y - 2)\mu^-(\sE) > d_Y + 2g-2.
$$
Suppose that $g=0$ or $1$. Since $(2d_Y - e-3+m)\mu^-(\sE) \geq 0$, we only need to check that $d_Y+3 > 4g$ but it is clear. Suppose that $g \geq 2$. It was shown in  \cite[Proof of Lemma 3.5]{KP2} that
$$
(d_Y - e-2)\mu^-(\sE) > 2g-2.
$$
Thus we have
$$
(2d_Y - e- 4)\mu^-(\sE)> d_Y + 4g-4.
$$
As $\mu^-(\sE) > (d_Y + 2g-2)/(d_Y-2) > 1$, we obtain
$$
(2d_Y - e-3+m)\mu^-(\sE) > (2d_Y - e- 4)\mu^-(\sE)+1 > d_Y + 4g-3
$$
for $m \geq 1$. Hence (\ref{eq:H^1((d-e-2mu+m)H_Y - (K_Y + nH_Y + L))}) holds.    
\end{proof}

\begin{lemma}
If $Y \subseteq \nP^r$ is a Roth variety but not a rational scroll, then  $n=2$ and $Y \in \lvert bH_S + \pi_S^* \sO_{\nP^1}(1) \rvert$ for some $b \geq 2$, where $S:=\nP (\sO_{\nP^1}^{\oplus 2} \oplus \sO_{\nP^1}(a))$ for some integer $a \geq 2$ with the tautological divisor $H_S$ and $\pi_S \colon S \to \nP^1$ is the canonical projection. In this case, $d_Y = ba+1$, and $\ell:=\sC(Y)$ is a line in $\nP^{r-1}$ with $\ell = H_Y - aF_Y$, where $F_Y$ is a fiber of the fibration $\pi_Y:=\pi_S|_Y \colon Y \to \nP^1$.
\begin{enumerate}
    \item If $b=2$, then (\ref{eq:H^1((d-e-2mu+m)H_Y - (K_Y + nH_Y + L))}) holds. 
    \item If $b \geq 3$, then 
    \begin{equation}\label{eq:H^1(widetildeX)}
    H^1(\widetilde{X}, \sigma^*D_{\inn}-E)=0,
    \end{equation}
    which directly implies that $r_{\widetilde{X}}^0$ is surjective. 
\end{enumerate}
\end{lemma}

\begin{proof}
The fibers of $\pi_Y$ are hypersurfaces of degree $b \geq 2$ in the fibers $\nP^n$ of $\pi_S$. If $n \geq 3$, then $\Pic(Y)$ is generated by $H_Y$ and $F_Y$. In this case, there is no linear subspace $L$ of dimension $n-1$ in $Y$. Thus $n=2$. Now, we only need to check the assertions (1) and (2). To this end, note that 
$$
K_Y = (K_S +Y)|_Y=(b-3)H_Y + (a-1)F_Y.
$$
The double point divisor from outer projection for $Y \subseteq \nP^{r-1}$ is
$$
-K_Y + (d_Y-4)H_Y = (a-1)(bH_Y - F_Y),
$$
which is nef. Since
$$
(bH_Y-F_Y)^2 = b^2(ba+1) - 2b^2 = b^2(ba-1) > 0,
$$
it follows that $bH_Y - F_Y$ is nef and big. 

\medskip

Suppose that $b=2$. Note that
$$
-K_Y = H_Y - (a-1)F_Y = \frac{1}{2a-1}(2H_Y - F_Y) + \frac{2a-3}{2a-1} \ell
$$
and $(Y, (2a-3)/(2a-1)\ell)$ is a klt Fano pair, i.e., it is a klt (Kawamata log terminal) pair and $-K_Y - (2a-3)/(2a-1)\ell$ is ample. Notice that $d_Y = ba+1 \geq a+3 \geq e+2$. As $H_Y - \ell$ and $H_Y - L$ are nef by Lemma \ref{lem:H-Lnef}, the divisor
$$
(d-e-2\mu+m)H_Y - (K_Y + nH_Y + L) = -K_Y + (d_Y - 4)H_Y + ((d_Y - e-1)H_Y - \ell) + (mH_Y - L)
$$
is nef for $m \geq 1$. Thus Kawamata--Viehweg vanishing yields (\ref{eq:H^1((d-e-2mu+m)H_Y - (K_Y + nH_Y + L))}). 

\medskip

Suppose that $b \geq 3$. Notice that 
$$
\ell^2 = (H_Y - aF_Y)^2 = (ba+1)-2ab = -ab+1 \leq -5.
$$
However, $\widetilde{X}|_{\tau_F^* L} = (2H_F +\tau_F^* L)|_{\tau_F^*L}$ is an effective divisor on $\tau_F^*L = \nP(\sO_{\nP^1} \oplus \sO_{\nP^1}(1))$, so $L^2 \geq -2$ (we will show that $L^2=-2$, which then implies that $\tau^*L= \widetilde{X}|_{\tau_F^*L} = 2E$). Thus $L \neq \ell$. As $-2 \leq L^2 \leq 0$ by Lemma \ref{lem:H-Lnef}, we have $-2 \leq K_Y.L \leq 0$. If $L.F_Y \geq 1$, then $K_Y.L \geq (b-3)+(a-1) \geq 1$, which is impossible. Thus $L.F_Y \leq 0$. This means that $L$ is contained in a fiber of $\pi_Y$. We may think that $L$ is a line and $F_Y - L$ is a (possibly reducible) plane curve of degree $b-1$ in $\nP^2$. We have
$$
0=L.F_Y = L.(L + (F_Y-L))=L^2 + L.(F_Y-L) = L^2 + (b-1),
$$
so $L^2=-b+1 \leq -2$. This implies that $L^2=-2$ and $b=3$. Now, applying Kodaira vanishing, we will show (\ref{eq:H^1(widetildeX)}). Note that $\widetilde{H} = \tau^*H_Y + E$ and $\tau^* L = 2E$. As $d=2(3a+1)+1=6a+3$ and $K_Y = (a-1)F_Y$, we may write
$$
\sigma^* D_{\inn} - E = (d-e-2)\widetilde{H} - \tau^* (K_Y + 2H_Y + L) - E = (6a-e-1)\tau^* H_Y + (6a-e-2)E - (a-1)\tau^* F_Y.
$$
Here $e \leq a+1$. Recall that $-K_{\widetilde{X}} = -\tau^*(K_Y + H_Y + L)$. Then
$$
\sigma^* D_{\inn} - E - K_{\widetilde{X}} = (6a-e-2)\tau^* H_Y + (6a-e-4)E - (2a-2)\tau^* F_Y.
$$
Notice that $2H_Y - F_Y + L = H_Y +(a-1)F_Y +\ell+ L$ is nef since
$$
(2H_Y - F_Y + L).\ell = 2-3+1=0~~\text{ and }~~(2H_Y - F_Y + L).L = 2-0-2=0.
$$
Thus the divisor
$$
\sigma^* D_{\inn} - E - K_{\widetilde{X}} = \underbrace{(2a-2)(2\tau^*H_Y - \tau^* F_Y + 2E)}_{\text{nef}} + \underbrace{(2a-e+2)\tau^* H_Y + (2a-e)E}_{=2\widetilde{H} + (2a-e)\tau^* H_Y~\text{is ample}}
$$
is ample, so (\ref{eq:H^1(widetildeX)}) holds by Kodaira vanishing. 
\end{proof}

We have completed the proof of Theorem \ref{thm:bpf}. 

\subsection{The case where the inner projection is singular}
As in Theorem \ref{thm:bpf}, suppose that $\sC(X)$ is a nonempty finite set and consider $\overline Y:=\overline{\pi_{x,X}(X \setminus \{x\})}$ for $x \in \sC(X)$. 
It would be interesting to explore what happens if $\overline Y$ is singular. We expect that this would lead us to classify all the cases where $D_{\inn}$ is not base point free.

\begin{question}
What can we say about the case when $\overline Y$ is singular? Can we classify all the cases where $D_{\inn}$ is not base point free?
\end{question}

We briefly explain how to construct an example where $\overline{Y}$ is singular.\footnote{We created and tested several examples using \texttt{Macaulay2} \cite{GS}, and in all of them, $D_{\inn}$ is base point free.}

\begin{example}\label{ex:singularY}
We use Noma's result \cite[Theorem 5]{N2} (see also Subsection~\ref{subsec:birdiv}) to construct an example where $\overline{Y}$ is singular. The rough idea is as follows. First, fix $\overline Y \subset \nP^{r-1}$ with a singularity. For computational simplicity, we take $\overline Y$ to be a hypersurface. Let $x \in \nP^r$ be a point, and $\pi_{x,\nP^r} \colon \nP^r \dashrightarrow \nP^{r-1}$ be a projection centered at $x$. In the cone $\overline F := \overline{\pi_{x,\nP^r}^{-1}(\overline Y)}$, we want to find a Weil divisor $X$, which is a smooth variety.

\medskip\par
To illustrate explicit equations, fix $r=4$ and $n = \dim X = 2$. Take a homogeneous coordinate $[Z_0,Z_1,Z_2,Z_3,Z_4]$ on $\nP^4$ such that $x = [0,0,0,0,1]$. Assume that $\overline Y $ is singular at $[Z_0,Z_1,Z_2,Z_3] = [0,0,0,1] \in \nP^3$. We find the ideal of $X$ based on the observation that $X$ is a Weil divisor on $\overline F$, which is Cartier except at $x$. Suppose $\bar Y$ is defined by an equation
\[
    g(Z_0,Z_1,Z_2,Z_3) = Z_0 g_0(Z_0,Z_1,Z_2,Z_3) + Z_1 g_1(Z_0,Z_1,Z_2,Z_3)
\]
with $\nabla g(0,0,0,1) = 0$. To fulfill $x \in X$, we take an equation which defines $X$ at general points:
\[
    f(Z_0,Z_1,Z_2,Z_3,Z_4) = \sum_{i=1}^\mu Z_4^{\mu-i} f_i(Z_0,Z_1,Z_2,Z_3).
\]
Here, $\mu \geq 3$ and $f_i \in \nC[Z_0,Z_1,Z_2,Z_3]$ is a homogeneous polynomial of degree $i$. Regarding $g$ as an element of $\nC[Z_0,Z_1,Z_2,Z_3,Z_4]$ we may consider the variety, say $X_I$ defined by the ideal $I := ( f,g )$. We are hoping to have a minimal prime ideal of $I$ which defines the desired $X$. To achieve it, we require one more assumption that $f_i \in ( Z_0, g_1)$ for each $i$. If $f_i = Z_0 A_i + g_1 B_i$ for some $A_i,B_i \in \nC[Z_0,Z_1,Z_2]$, then
\[
    f_i' := Z_1 A_i - g_0 B_i
\]
is the polynomial which corresponds to ``$\tfrac{Z_1}{Z_0} f_i$'', since $\frac{Z_1}{Z_0} + \frac{g_0}{g_1} = 0$ in $\overline F$, and this will belong to the desired minimal prime of $I$.

\medskip\par The following is one of the explicit example which gives us a smooth $X$ with $x \in \sC(X)$. Let $g_0(Z_0,Z_1,Z_2,Z_3) = 3  Z_2  Z_3$ and 
\[
    g_1(Z_0,Z_1,Z_2,Z_3) =   4 Z_0^2+9 Z_0 Z_1+4 Z_0 Z_2+8 Z_2^2+4 Z_0 Z_3+8 Z_1 Z_3
\]
and let $g = Z_0 g_0 + Z_1 g_1$. Then, $\bar Y = V(g)$ is the cubic surface with one ordinary double point at $[0,0,0,1]$. Let
\begin{align*}
f(Z_0,\dots,Z_4) ={}&
 Z_4^2  Z_1 + Z_4( 9   Z_0  Z_1 + 9 Z_1^2 +8  Z_1  Z_2 + 8 Z_1  Z_3 )  \\
 &{}  + 5  Z_0^2  Z_1 + 5  Z_0  Z_1^2 + 5  Z_1^3 + 8  Z_0  Z_1  Z_2+9  Z_1^2  Z_2+3  Z_1  Z_2^2 \\
 &{}+7  Z_1^2  Z_3 +5  Z_0  Z_2  Z_3 +3  Z_1  Z_2  Z_3+8  Z_2^2  Z_3+7  Z_1  Z_3^2+4  Z_2  Z_3^2
\end{align*}
The minimal primes over $(f,g)$ are $(Z_1,Z_2)$, $(Z_1,Z_3)$ and the ideal of $X$ with the desired properties.
\end{example}

%%%%%%%%%%%%%%%%%%%%%%%%%%%%%%%%%%%%%%%%%%%%%%%%%%%%%%%
\section{Bigness of $D_{\inn}$}\label{sec:big}
%%%%%%%%%%%%%%%%%%%%%%%%%%%%%%%%%%%%%%%%%%%%%%%%%%%%%%%
In this section, we investigate when $D_{\inn}$ is big, and prove Theorems \ref{thm:D_innbig} and \ref{thm:D_innbigdim=3} and Corollary \ref{cor:Iitakadim}. We assume that $n \geq 2$ and $X \subseteq \nP^r$ is not a Noma's exceptional variety.

\subsection{The case when $D_{\inn}$ is not big}
We first show that if $D_{\inn}$ is not big, then $D_{\inn}$ is base point free. The following lemma is probably well-known. 

\begin{lemma}\label{lem:baselocus}
Let $Z$ be a smooth projective variety of dimension $n$, and $D$ be a nonzero semiample divisor on $Z$. If $b:=\dim \Bs \lvert D \rvert \geq 0$, then $\kappa(D) \geq n-b$.
\end{lemma}

\begin{proof}
Suppose that $\kappa(D) \leq n-b-1$.
Let $f \colon Z \rightarrow W$ be the morphism induced by the complete linear system $\lvert mD \rvert$ for a sufficiently large and divisible integer $m \gg 0$, where $W$ is a projective variety of dimension $\kappa(D)$. Choose a point $z \in \Bs \lvert D \rvert$. Notice that $\dim f^{-1}(f(z)) \geq n-\kappa(D) \geq b+1$. Thus we can take a point $z' \in f^{-1}(f(z)) \setminus \Bs \lvert D \rvert$. Now, there is an effective divisor $D' \in |D|$ such that $z' \not\in \Supp(D')$. On the other hand, $z \in \Supp(D')$. Note that $mD' = f^*A$ for some nonzero effective divisor $A$ on $W$. We have $f(z) \in \Supp(A)$. Then 
$$
w' \in f^{-1}(f(z)) \subseteq f^{-1}(\Supp(A)) = \Supp(mD') = \Supp(D'),
$$
which gives a contradiction. Hence $\kappa(D) \geq n-b$.
\end{proof}

\begin{proposition}\label{prop:baselocus}
If $D_{\inn}$ is not big, then $D_{\inn}$ is base point free. 
\end{proposition}

\begin{proof}
Suppose that $D_{\inn}$ is not base point free. In particular,  $D_{\inn} \neq 0$. We know that $\Bs\lvert D_{\inn} \rvert \subseteq \sC(X)$ and $\dim \sC(X) \leq 0$, and thus, $\dim \Bs\lvert D_{\inn} \rvert=0$. By Lemma \ref{lem:baselocus}, $\kappa(D_{\inn}) \geq n$, i.e., $D_{\inn}$ is big.
\end{proof}

\begin{remark} 
Proposition \ref{prop:baselocus} also follows from Noma's recent result \cite[Theorem 1.2 (2)]{N5} asserting that if $X \subseteq \nP^r$ is not a Roth variety and $\sC(X) \neq \emptyset$, then $D_{\inn}^n \geq \lvert \sC(X) \rvert$.
Indeed, if $D_{\inn}$ is not big, then $D_{\inn}^n = 0$ and thus $\sC(X) = \emptyset$. Since $\Bs\lvert D_{\inn} \rvert \subseteq \sC(X)$, it follows that $\lvert D_{\inn} \rvert$ is base point free.
\end{remark}

As for adjunction mappings, it would be an interesting problem to explore the structure of the morphism $\varphi \colon X \to Y$ given by $\lvert D_{\inn} \rvert$ when $D_{\inn}$ is not big. Here we exhibit three examples where $D_{\inn}$ is not big. We are not aware any other examples.

\begin{example}\label{ex:nonbig}
$(1)$  By \cite[Lemma 2.5]{KP2}, $D_{\inn}=0$ if and only if $X \subseteq \nP^r$ is a \emph{del Pezzo variety}, by which we mean a linearly normal smooth projective variety with $d=e+2$. It is well-known that if $X \subseteq \nP^r$ is a del Pezzo variety, then $-K_X = (n-1)H$.\\[3pt]
$(2)$ Let $S:=\nP(\sO_{\nP^1} \oplus \sO_{\nP^1}(1)^{\oplus n})$ for $n \geq 3$ with the tautological divisor $H_S$ and a fiber $F_S$. Take a general smooth member $X \in |2H_S|$, and put $H:=H_S|_X, F:=F_S|_X$.  Note that $\lvert H \rvert$ gives an embedding of $X$ into $\nP^{2n}$. Thus $X \subseteq \nP^{2n}$ is a non-degenerate smooth projective variety of dimension $n$, codimension $n$, and degree $d=H_S^{n}.2H_S=2n$. Clearly, $X \subseteq \nP^{2n}$ is not a Noma's exceptional variety. 
Note that
$$
K_S=-(n+1)H_S + (n-2)F_S~~\text{ and }~~K_X = (K_S+2H_S)|_X= -(n-1)H + (n-2)F.
$$
Thus
$$
D_{\inn}=-K_X - H = (n-2)(H-F).
$$
Since
$$
D_{\inn}^n = (n-2)^n(H_S-F_S)^n.2H_S = 2(n-2)^n(H_S^{n+1}-nH_S^{n}.F_S) = 0,
$$
it follows that $D_{\inn}$ is not big. The morphism induced by $\lvert D_{\inn} \rvert$ is a conic fibration over $\nP^{n-1}$.\\[3pt]
$(3)$ Let $S:=\nP(\sO_{\nP^1}(1)^{\oplus n+1})$ for $n \geq 4$ with the tautological divisor $H_S$ and a fiber $F_S$. Take a general smooth member $X \in |2H_S-2F_S|$, and put $H:=H_S|_X$ and $F:=F_S|_X$. Note that $|H|$ gives an embedding of $X$ into $\nP^{2n+1}$. Thus $X \subseteq \nP^{2n+1}$ is a non-degenerate smooth projective variety of dimension $n$, codimension $n+1$, and degree $d=H_S^n(2H_S-2F_S)=2n$. Clearly, $X \subseteq \nP^{2n+1}$ is not a Noma's exceptional variety.
Note that 
$$
K_S=-(n+1)H_S + (n-1)F_S~~\text{ and }~~K_X = (K_S+2H_S-2F_S)|_X = (-n+1)H + (n-3)F.
$$
Thus
$$
D_{\inn}=-K_X - 2H =(n-3)(H-F).
$$
Since
$$
D_{\inn}^n = (n-3)^n(H_S-F_S)^n(2H_S-2F_S)=2(n-3)^n(H_S^{n+1}-(n+1)H_S^nF_S) = 0,
$$
it follows that $D_{\inn}$ is not big. The morphism induced by $\lvert D_{\inn} \rvert$ is a linear fibration over $Q^{n-1}$. In fact, $X = \nP^1 \times Q^{n-1}$.
\end{example}

\subsection{Overall framework}
In the remaining, we show that $D_{\inn}$ is big in most cases. We assume that $X \subseteq \nP^r$ is neither a Noma's exceptional variety nor a del Pezzo variety (cf. Example \ref{ex:nonbig} $(1)$). In particular, $d \geq e+3$. Recall that if $X \subseteq \nP^r$ is not linearly normal, then $D_{\inn}$ is ample. As our aim is to show $D_{\inn}$ is big under some assumptions, we may henceforth assume that $X \subseteq \nP^r$ is linearly normal. Now, we introduce a technique developed in \cite{KP1, KP2} for proving $D_{\inn}$ is big. Let
$$
g:=\frac{1}{2}(K_X.H^{n-1} + (n-1)H^n) + 1
$$ 
be the sectional genus of $X \subseteq \nP^r$.

\begin{lemma}\label{lem:D_innbig}
If
\begin{equation}\label{eq:D_innbig}
\frac{(d-e-2)d + (\varepsilon+1)(e+1-\varepsilon)}{e} < \frac{(d-e-2)d+2n}{n},
\end{equation}
then $D_{\inn}$ is big.
\end{lemma}

\begin{proof}
As $d \geq e+3$, the divisor $D:=(d-e-2)H$ is very ample. By \cite[Theorem 1]{Io1}, $E:=K_X + (n-1)H$ is base point free under our assumption. By Siu's bigness criterion (cf. \cite[Theorem 2.2.15]{positivity}), if 
\begin{equation}\label{eq:DDE}
 n D^{n-1}.E < D^n,
\end{equation}
then $D-E = D_{\inn}$ is big . As $2g=K_X.H^{n-1}+(n-1)H^n$, we may rewrite (\ref{eq:DDE}) as
\begin{equation}\label{eq:DDE2}
2g < \frac{(d-e-2)d + 2n}{n}.
\end{equation}
Applying Castelnuovo's genus bound (\ref{eq:Castelnuovo-big}), we see that (\ref{eq:D_innbig}) implies (\ref{eq:DDE2}).
\end{proof}

\begin{lemma}\label{lem:D_innd<=r}
If $d \leq r$, then $D_{\inn}$ is big except when either $n \geq 3$, $d=2n$, $e=n$ and $\lvert D_{\inn} \rvert$ induces a conic fibration $X \to \nP^{n-1}$ (cf. Example \ref{ex:nonbig}), or $n \geq 4$, $d=2n$, $e=n+1$ and $\lvert D_{\inn} \rvert$ induces a linear fibration $X \to Q^{n-1}$ (cf. Example \ref{ex:nonbig}).
\end{lemma}

\begin{proof}
By Ionescu's classification \cite{Io3} of smooth projective varieties $X \subseteq \nP^r$ of degree $d \leq r$, one can easily check the assertion.
\end{proof}

\begin{proposition}\label{prop:D_innbig}
Suppose that one of the following hold:
\begin{enumerate}
 \item[$(1)$] $e \geq n+1$ and either $\varepsilon \geq n-1$ or $\varepsilon \leq e-n+1$, or
 \item[$(2)$] $e=n$ and $\varepsilon = 0$.
\end{enumerate}
Then \emph{(\ref{eq:D_innbig})} holds, i.e., $D_{\inn}$ is big, unless $n \geq 4, d=2n, e=n+1$ and $\lvert D_{\inn} \rvert$ induces a linear fibration $X \to Q^{n-1}$ (cf. Example \ref{ex:nonbig}).
\end{proposition}

\begin{proof}
We may rewrite (\ref{eq:D_innbig}) in  Lemma \ref{lem:D_innbig} as
\begin{equation}\label{eq:DDE3}
n(\varepsilon-1)(e-\varepsilon-1) < (e-n)(d-e-2)d.
\end{equation}
If the condition $(2)$ holds, then (\ref{eq:DDE3}) holds. Now, suppose that the condition $(1)$ holds.
First, assume that $\varepsilon \geq n-1$. Then $e-\varepsilon-1 \leq e-n$. In this case, $\epsilon -1 \leq d-e-2$ and $n \leq e <  d$ so that  (\ref{eq:DDE3}) holds.
Next, assume that $\varepsilon \leq e-n+1$. Then $\varepsilon-1 \leq e-n$. If $m \geq 2$, then $e-\varepsilon-1 \leq d-e-2$ and $n \leq e < d$ so that  (\ref{eq:DDE3}) holds. Suppose that $m=1$. Then $d=e+\varepsilon+1$. We may assume that $\varepsilon \leq n-2$ (if not, (\ref{eq:DDE3}) holds). Then $d \leq r-1$, so the assertion follows from Lemma \ref{lem:D_innd<=r}.
\end{proof}

\subsection{Proofs of Theorems \ref{thm:D_innbig} and \ref{thm:D_innbigdim=3} and Corollary \ref{cor:Iitakadim}}

\begin{proof}[Proof of Theorem \ref{thm:D_innbig}]
By Lemma \ref{lem:D_innd<=r}, we may assume that $d \geq r+1$. We need to show that if 
\begin{equation}\label{eq:d>=(2n-e-2)/(2e-2n)e+1}
d \geq \left\lceil \frac{2n-e-2}{2e-2n} \right\rceil e+1,
\end{equation}
then $D_{\inn}$ is big. We rewrite (\ref{eq:D_innbig}) in  Lemma \ref{lem:D_innbig} as
\begin{equation}\label{eq:DDE4}
q(\varepsilon):=\varepsilon^2 + \big(2m(e-n)-e\big)\varepsilon + m(m-1)(e-n)e+2n-e-1 > 0.
\end{equation}
The assumption (\ref{eq:d>=(2n-e-2)/(2e-2n)e+1}) means that
$$
m \geq \frac{2n-e-2}{2e-2n},
$$
which is equivalent to
\begin{equation}\label{eq:e-n+1}
e-n+1 \geq -\frac{1}{2} \big(2m(e-n)-e\big).
\end{equation}
Proposition \ref{prop:D_innbig} says (\ref{eq:D_innbig}) holds when $\varepsilon = e-n+1$. As (\ref{eq:D_innbig}) is equivalent to (\ref{eq:DDE4}), we have
\begin{equation}\label{eq:q(e-n+1)}
q(e-n+1) > 0.
\end{equation}
Notice that the quadratic function $q(\varepsilon)$ in $\varepsilon$ is increasing when 
$\varepsilon \geq -1/2 \big(2m(e-n)-e\big)$.
Then (\ref{eq:e-n+1}) and (\ref{eq:q(e-n+1)}) imply that $q(\varepsilon)>0$ for $\varepsilon \geq e-n+1$. On the other hand, by Proposition \ref{prop:D_innbig}, we have $q(\varepsilon)>0$ for $\varepsilon \leq e-n+1$. Hence (\ref{eq:DDE4}) holds for all $\varepsilon$.
\end{proof}

\begin{proof}[Proof of Theorem \ref{thm:D_innbigdim=3}]
$(1)$ Assume $n=2$. By Proposition \ref{prop:D_innbig}, we only have to consider the case of $e=2$ and $\varepsilon = 1$. 
By Castelnuovo's genus bound (\ref{eq:Castelnuovo-big}), we have 
$$
2g \leq \frac{1}{2} d^2- 2d + 2.
$$
If the inequality is strict, then (\ref{eq:DDE2}) holds so that $D_{\inn}$ is big. Assume that the equality holds, i.e., $2g=(1/2)d^2-2d+2$. Then a general curve section $C \subseteq \nP^3$ of $S \subseteq \nP^4$ is projectively normal and contained in a quadric surface. Since $d=2m+2$ is even, it follows that $C \subseteq \nP^3$ is a complete intersection (see e.g., \cite[page 119]{ACGH}). Thus $S \subseteq \nP^r$ is also a complete intersection. In this case, since $S \subseteq \nP^r$ is not a linearly normal del Pezzo surface, it follows that $m \geq 2$, i.e., $d \geq 6$. Then $D_{\inn}$ is very ample. 

\medskip

\noindent $(2)$ Assume $n=3$. If $e \geq 4$ or $e=3$ and $\varepsilon=0$, then $D_{\inn}$ is big by Proposition \ref{prop:D_innbig}. Suppose $e=3$ and $\varepsilon = 1$ or $2$. Then (\ref{eq:DDE3}) is the equality, so (\ref{eq:D_innbig}) is the equality. If (\ref{eq:Castelnuovo-big}) is the strict inequality, then (\ref{eq:DDE}) holds so that $D_{\inn}$ is big by Siu's bigness criterion. Thus we assume that (\ref{eq:Castelnuovo-big}) is also the equality, i.e., $6g=d^2-5d+6=(d-2)(d-3)$. If $d \leq 8$, then $(g,d)=(2,6)$ or $(5,8)$. By Ionescu's classification \cite{Io1, Io2} of smooth projective varieties $X \subseteq \nP^r$ of degree $d \leq 8$, we can easily check the assertion for $d \leq 8$. Assume that $d \geq 9$. Since a general curve section $C \subseteq \nP^4$ of $X \subseteq \nP^6$ is a Castelnuovo curve, it follows that $X$ is a smooth divisor on a rational normal scroll (see \cite[Theorem 2.5]{ACGH}). As $\dim \sC(X) \leq 0$ (actually, we may assume $\sC(X) = \emptyset$ by Proposition \ref{prop:baselocus}, but we don't need this), \cite[Proposition 5.2]{N2} shows that $X \in |aH_S|$ for some integer $a \geq 3$, where $S:=\nP (\sO_{\nP^1} \oplus \sO_{\nP^1}(1)^{\oplus 3})$ with the tautological divisor $H_S$. Here $d=3a$ and $K_X = (a-4)H + F$, where $F$ is the restriction of a fiber of $S$ over $\nP^1$ to $X$. Then 
$$
D_{\inn} = -K_X + (3a-7)H = (2a-3)H-F,
$$
so
$$
D_{\inn}^3 = (2a-3)^3 3a - 3(2a-3)^2a = 3(2a-3)^2a(2a-4) > 0.
$$
Hence $D_{\inn}$ is big.

\medskip
Finally, consider the case of $e=2$. Notice that $X \subseteq \nP^5$ is linearly normal and $H^1(X, \sO_X)=0$. If $X$ is contained in a quadric hyperusrface in $\nP^5$, then $X \subseteq \nP^5$ is a complete intersection or a Roth variety by \cite[Corollary 2.7]{Kwak} and \cite[Proposition 5.2]{N2}. In the former case, $D_{\inn}$ is very ample. We assume that $X$ is not contained in a quadric hypersurface in $\nP^5$. If $C \subseteq \nP^3$ is a general curve section of $X \subseteq \nP^5$, then $C$ is not contained in a quadric hypersurface in $\nP^3$ as $X \subseteq \nP^5$ is linearly normal. By Halphen's genus bound (see e.g., \cite[1.6]{Ottaviani}), we have
\begin{equation}\label{eq:Halphen}
\frac{d^2-3d+6}{6} \geq g,
\end{equation}
where $g$ is the genus of $C$. Note that $\lvert K_X + 2H \rvert$ is base point free (see \cite[Theorem 1.4]{Io1}). If the adjunction mapping of $X$ gives a hyperquadric fibration over a curve, then $X$ is a smooth divisor of a rational normal scroll. In this case, $X$ is contained in a quadric hypersurface in $\nP^5$. If the adjunction mapping of $X$ gives a linear fibration over a surface, then one can easily check that $D_{\inn}$ is big using Ottaviani's classification \cite{Ottaviani} (see especially Table 1 in page 460). Thus we may assume that $K_X + 2H$ is big by adjunction theory (cf. \cite[Proposition 1.11]{Io1}). If $d \leq 7$, then Ionescu's classification \cite{Io1, Io2} shows that $D_{\inn}$ is big. Thus we may assume that $d \geq 8$. By \cite[Theorem 12.3.1]{BS}, $K_X+2H$ is ample. Then \cite[Theorem 7.3.4]{BS} implies that $E:=K_X + H$ is nef. Note that $D:=(d-5)H$ is very ample. By Siu's bigness criterion, $D_{\inn}=D-E$ is big if $3 D^2.E < D^3$, which is equivalent to $3(2g-2-d) < (d-5)d$. But this follows from (\ref{eq:Halphen}).
\end{proof}

\begin{corollary}\label{cor:e>=(4n-2)/3=>D_innbig}
If $e \geq (4n-2)/3$, then $D_{\inn}$ is big except when $n=4,5$, $e=n+1$, $d=2n$ and $\lvert D_{\inn} \rvert$ induces a linear fibration $X \to Q^{n-1}$ (cf. Example \ref{ex:nonbig}).
\end{corollary}

\begin{proof}
When $n \leq 3$, the assertions follow from Theorem \ref{thm:D_innbigdim=3}. We henceforth assume that $n \geq 4$. Note that $e \geq (4n-2)/3$ implies $e \geq n+1$ when $n \geq 4$ and $e=n+1$ if and only if $n=4,5$. Furthermore, $e \geq (4n-2)/3$ is equivalent to $(2n-e-2)/(2e-2n) \leq 1$. Then the corollary is a consequence of Theorem \ref{thm:D_innbig}.
\end{proof}

\begin{proof}[Proof of Corollary \ref{cor:Iitakadim}]
If $e \leq n-2$, then Barth--Larsen theorem implies that $D_{\inn}$ is very ample. Thus we assume that $e \geq n-1$. By Lemma \ref{lem:D_innd<=r}, the assertion holds when $d \leq r$. Thus we also assume that $d \geq r+1$. Now, since $D_{\inn}$ is semiample, it follows that $\kappa(D_{\inn})=\nu(D_{\inn})$, where $\nu(D_{\inn})$ is the numerical Iitaka dimenson. It suffices to show that $D_{\inn}^{n-k} \not\equiv 0$ for $n-k \leq (3n-1)/4$. If $Y \subseteq \nP^{r-k}$ is a general $(n-k)$-dimensional section of $X \subseteq \nP^r$, then Corollary \ref{cor:e>=(4n-2)/3=>D_innbig} implies that $D_{\inn}|_Y$ is big as soon as $n-1 \geq (4(n-k)-2)/3$, which is equivalent to $(3n-1)/4 \geq n-k$. Thus we can conclude that $D_{\inn}^{n-k} \not\equiv 0$ for $(3n-1)/4 \geq n-k$. 
\end{proof}

This paper ends with the following question.

\begin{question}
Can we classify all the cases where $D_{\inn}$ is not big? In each case, what can we say about the structure of the morphism given by $\lvert D_{\inn} \rvert$?
\end{question}

%%%%%%%%%%%%%%%%%%%%%%%%%%%%%%%%%%%%%%%%%%%%%%%%%%%%%%%%%%%%%%%%%%%%%%%%%%%%%%%%%%%%%%%%%%%%%%%%%%%%%%%%

\end{document}